\documentclass[11pt,reqno]{amsart}
\usepackage[margin=1in]{geometry}

\usepackage{amsmath, amssymb, amsthm, tikz}
\usetikzlibrary{positioning}

\usepackage{hyperref}

\usepackage[noabbrev,capitalize]{cleveref}
\crefname{equation}{}{}

\newtheorem{theorem}{Theorem}
\newtheorem{lemma}[theorem]{Lemma}
\newtheorem*{lemma*}{Lemma}
\newtheorem{corollary}[theorem]{Corollary}
\newtheorem*{corollary*}{Corollary}

\newtheorem*{theorem*}{Theorem}

\newtheorem{conj}[theorem]{Conjecture}
\newtheorem{nconj}[theorem]{Na{\"i}ve Conjecture}
\newtheorem{prop}[theorem]{Proposition}

\theoremstyle{definition}
\newtheorem{example}[theorem]{Example}
\newtheorem{definition}[theorem]{Definition}

\theoremstyle{remark}
\newtheorem*{remark}{Remark}

\newcommand\R{\mathbb{R}}\newcommand\RR{\mathbb{R}}

\newcommand\E{\mathbb{E}}
\newcommand\EE{\mathbb{E}}
\newcommand\PP{\mathbb{P}}

\newcommand{\paren}[1]{\left( #1 \right)}

\newcommand{\set}[1]{\left\{ #1 \right\}}

\newcommand{\wt}{\widetilde}

\newcommand{\eps}{\varepsilon}

\newcommand{\GW}{\textbf{\textsc{GW}}}

\newcommand{\G}{\mathcal{G}}
\newcommand{\Vol}{\mathrm{Vol}}
\newcommand{\DiscLip}{\mathrm{DiscLip}}
\newcommand{\DiscDerv}{\mathrm{DiscDerv}}
\newcommand{\DiscGrad}{\mathrm{DiscGrad}}
\newcommand{\DiscGW}{\mathrm{DiscGW}}

\renewcommand{\d}{\delta}
\newcommand{\bs}{\backslash}
\newcommand{\D}{\Delta}
\newcommand{\ls}{\lesssim}
\newcommand{\gs}{\gtrsim}

\renewcommand{\l}{\langle}
\renewcommand{\r}{\rangle}
\renewcommand{\forall}{\text{ for all }}

\newcommand{\nbone}{\overline{B_1^{(i)}}}

\newcommand{\xr}{x_{[r]}}
\newcommand{\oB}{\overline{B}}
\renewcommand{\S}{\mathcal{S}}
\newcommand{\dx}{\mathrm d x}
\newcommand{\dy}{\mathrm d y}
\newcommand{\dz}{\mathrm d z}
\newcommand{\ka}{\kappa}
\newcommand{\h}{\H}

\numberwithin{theorem}{section}

\author[Liu]{Yang P. Liu}
\address{Liu: Department of Mathematics, Stanford University, Stanford, CA, USA}
\email{yangpliu@stanford.edu}

\author[Zhao]{Yufei Zhao}
\address{Zhao: Department of Mathematics, Massachusetts Institute of Technology, Cambridge, MA, USA}
\email{yufeiz@mit.edu}

\thanks{Liu was supported by the Department of Defense (DoD) through the National Defense Science and Engineering Graduate Fellowship (NDSEG) Program. Zhao was supported by NSF Award DMS-1764176, the MIT Solomon Buchsbaum Fund, and a Sloan Research Fellowship.}

\begin{document}

\title[Upper tail for hypergraphs]{On the upper tail problem for random hypergraphs}
\begin{abstract}
The upper tail problem in a random graph asks to estimate the probability that the number of copies of some fixed subgraph in an Erd\H{o}s--R\'enyi random graph exceeds its expectation by some constant factor.
There has been much exciting recent progress on this problem. 

We study the corresponding problem for hypergraphs, for which less is known about the large deviation rate. We present new phenomena in upper tail large deviations for sparse random hypergraphs that are not seen in random graphs. We conjecture a formula for the large deviation rate, i.e., the first order asymptotics of the log-probability that the number of copies of fixed subgraph $H$ in a sparse Erd\H{o}s--R\'enyi random $k$-uniform hypergraph exceeds its expectation by a constant factor. This conjecture turns out to be significantly more intricate compared to the case for graphs. We verify our conjecture when the fixed subgraph $H$ being counted is a clique, as well as when $H$ is the 3-uniform 6-vertex 4-edge hypergraph consisting of alternating faces of an octahedron, where new techniques are required.
\end{abstract}

\maketitle

\section{Introduction}

\subsection{The upper tail problem in random graphs}

Given a fixed graph $H$, the ``infamous upper tail'' problem, a name given by Janson and Rucinski \cite{JR02}, asks to estimate the probability that the number of copies of $H$ in an Erd\H{o}s--R\'enyi random graph exceeds its mean by some given constant factor. This problem has played a central role in the development of probabilistic combinatorics, and had led to the development of a host of useful concentration inequalities. There were a number of significant advances on this problem in the past decade or so. We begin by summarizing some of the recent developments. 

Let $X_H$ denote the number of copies of $H$ in an Erd\H{o}s--R\'enyi random graph $G_{n,p}$. A problem of great interest is to estimate the probability that $X_H \ge (1+\delta) \E X_H$, where $\delta > 0$ is fixed but $p = p(n)$ is allowed to vary with $n$.

Even the order of the log-probability had resisted much attack, until it was determined independently by Chatterjee \cite{Cha12} and DeMarco and Kahn \cite{DK12}. Once the order of log-probability had been determined, the attention turns to pinning down the leading constant, i.e., the first order asymptotics of the log-probability of upper tails. 

As is commonly the case with large deviation problems, there are two complementary steps:
\begin{enumerate}
\item[(1)] Developing a large deviation principle/framework that reduces the rate problem to a natural variational problem over edge-weighted graphs or graphons;
\item[(2)] Solving the variational problem.
\end{enumerate}
For random graphs, neither step is easy. There were a number of recent breakthroughs that have lead to a satisfying understanding in many interesting cases, though there is still much mystery in general as well as in natural variations of the problem.

For clarity, let us focus on the case $H = K_3$. Progress towards Step (1), the development of the large deviation principle, began with the seminal paper of Chatterjee and Varadhan \cite{CV11}, which proves a large deviation principle for dense random graphs using Szemer\'edi's graph regularity lemma. Due to the poor quantitative dependencies in the graph regularity lemma, Chatterjee and Varadhan's method is applicable to random graphs $G_{n,p}$ with constant $p$ (i.e., dense graphs) or extremely slowly decreasing $p$, e.g., $p \ge (\log n)^{-c}$ (see \cite[Section 5]{LZ17} for an explanation how to apply the weak regularity lemma to derive this result). 
Subsequently, Chatterjee and Dembo \cite{CD16}, using ideas from Stein's method for exchangeable pairs, derived the first nonlinear large deviation principle that allows $p$ to decay as a power of $n$. For triangles, their theorem holds when $p \ge n^{-1/42}\log n$. Eldan later used a different method, namely, using stochastic differential equations to analyze a certain modified Brownian motion, improved the range of validity to $p \ge n^{-1/18}\log n$ in the case of triangles. Independent results by Cook--Dembo \cite{CD} and Augeri \cite{Aug} further improved the range of validity for $H = K_3$ to $p \gg n^{-1/3}$ and $p \gg n^{-1/2}$ respectively. Very recently, Harel, Mousset, and Samotij \cite{HMS}, using a novel combinatorial approach, resolved the problem for all ranges of $p$ in the case of triangle upper tails, though for general $H$ there remain gaps to be closed. A separate solution to the lower tail problem has been announced in a forthcoming work of Kozma and Samotij.

For Step (2), in the setting of dense random graphs ($p$ constant), one can ask whether the variational problem is optimized by a constant graphon (we say that ``replica symmetry" occurs when this happens), and this question was answered by Lubetzky and Zhao \cite{LZ15} for every regular graph $H$.  From now on, let us consider the sparse setting $p \to 0$. The variational problem was solved by Lubetzky and Zhao \cite{LZ17} in the case of $H = K_3$, and more generally, when $H$ is a clique. For general graphs $H$, although the case $H = C_4$ already presented a significant hurdle, the variational problem was solved for all $H$ by Bhattacharya, Ganguly, Lubetzky, and Zhao \cite{BGLZ17}. In contrast, the lower tail variational problem, studied in \cite{Zhao17}, has a completely different behavior, for which some basic questions are still open. Recently, the corresponding problem for random regular graphs was also studied~\cite{BD}.

Combining these developments, in particular~\cite{HMS,LZ17}, our knowledge of the upper tail rate for cliques $H=K_r$ is summarized below. See Section \ref{sec:notation} for asymptotic notation.
\begin{theorem} \label{thm:upper-tail-graphs-clique}
	Fix integer $k \ge 3$ and real $\delta > 0$. Let $X = X_{K_k}$ be the number of $k$-cliques in the random graph $G_{n,p}$, where $p = p(n)$. Then
	\[
		\lim_{n \to \infty} \frac{-\log \PP (X \ge (1+\delta) \EE X)}{n^2 p^{k-1} \log(1/p)} 
		= \begin{cases}
			\frac12\delta^{2/k} & \text{if } n^{-2/(r-1)} (\log n)^{\frac{2}{(r-2)(r-1)}} \ll p \ll n^{-1/(k-1)}, \\
			\min\{ \frac12\delta^{2/k}, \frac{\d}{k} \}
			& \text{if } n^{-1/(k-1)} \ll p \ll 1.
		\end{cases}
	\]
\end{theorem}
The lower bounds to the upper tail probability come from constructions where we plant either a clique on $\delta^{1/k} n p^{(k-1)/2} $ vertices or a hub on $\delta n p^{k-1}/k$ vertices (a \emph{hub} is a set of vertices each adjacent to all vertices of the graph). As shown recently in \cite{HMS}, these two constructions approximately describe the typical structure of the random graph conditioned on the upper tail event. We refer the readers to \cite{HMS} for precise descriptions of these results as well as much more general statements covering other settings.

\subsection{Random hypergraphs}

The aim of this paper is to initiate the study of the variational problem, i.e., Step (2) above, for the corresponding upper tail problem for random hypergraphs. Here we write $G_{n,p}^{(r)}$ for the random $r$-uniform hypergraph (or simply \emph{$r$-graph}) where every possible edge appears independently with probability $p$. We are interested in estimating upper tail probabilities of the number $X_H$ of copies of some fixed $r$-graph $H$ in this random hypergraph.

Some, but not all, of the developments of large deviation principles for random graphs (i.e., Step (1) above) transfer nicely over to the setting of hypergraphs. Proofs that involve the spectral data of a graph tend to encounter some difficulty as hypergraphs lack a useful spectral decomposition. On the other hand, Eldan's non-linear large deviation principle \cite{Eld18}, which uses the Gaussian width as a measure of complexity, transfer over nicely to hypergraphs, as we explain in the appendix (calculations of a similar spirit were done in \cite{BGSZ,BG}).

It seems likely that other recent breakthroughs on large deviations in random graphs, including \cite{CD,HMS}, might be adapted to the setting of hypergraphs, perhaps allowing the entire range of $p$, but this has yet to be worked out. In any case, even with an improved large deviation principle, one still needs to solve a  variational problem to determine the large deviation rate for the upper tail theorem for random hypergraphs, and even the form of the rate function appears to be highly non-trivial.

Thus we now turn our focus to the variational problem. It turns out, as we explain in this paper, the situation appears to be much more intricate than graphs, and it takes some effort to even state a reasonable conjecture. Our conjecture is essentially that the rate is obtained by planting a certain ```compatible'' and ``stable'' collection of boxes (which we call ``mixed hubs'') onto the random hypergraph. We shall motivate the formulation of the conjecture in several steps by explaining why some na{\"i}ve versions of the statement must be incorrect. In addition, we verify our conjecture in two different cases: 

\begin{enumerate}
\item[(1)] $H$ is a clique: this case is similar to clique counts in a random graph, which was solved in \cite{LZ17}
\item[(2)] $H$ is the 3-uniform hypergraph in \cref{fig:graph1}: this case already requires new techniques (multiscale thresholding) not present in earlier works.
\end{enumerate}

We study the 3-graph from \cref{fig:graph1} since it is the first interesting example of a hypergraph whose proof requires new methods, and its analysis already requires substantial work. Combining our knowledge of the large deviation principle and the solution of the variational problem, we state the following results.

\begin{theorem} \label{thm:rate}
	Let $r \ge 2$ and fix an $r$-graph $H$. Fix a real $\delta > 0$. Let $X_H$ denote the number of copies of $H$ in the random $r$-graph $G_{n,p}^{(r)}$, where $p = p(n)$ satisfies $p = o(1)$ and $p > n^{-\frac{1}{6|E(H)|}}\log n$. 
	\begin{enumerate}
		\item [(a)] If $H = K_k^{(r)}$ (clique on $k$ vertices), then
		\[
		\PP( X_H \ge (1+\delta) \EE X_H)
		=
		\exp\left(-(1+o(1)) \min\left\{ \frac{\delta^{r/k}}{r!} , \frac{\delta}{ (r-1)! k}  \right\} n^r p^{\binom{k-1}{r-1}} \log(1/p) \right).
		\]
		\item[(b)] If $H$ is the 3-graph in \cref{fig:graph1}, then 
		\[
		\PP( X_H \ge (1+\delta) \EE X_H)
		=
		\exp\left(-\paren{\frac{1}{6} + o(1)} \min\left\{ \sqrt{9+3\delta} - 3, \sqrt{\delta} \right\}  n^3 p^2 \log(1/p) \right).
		\]
	\end{enumerate}
\end{theorem}

\begin{figure}
\begin{equation*}
\begin{tikzpicture}[scale=.8, every node/.style={draw, circle, black, fill, inner sep = 0pt, minimum width = 3pt},font=\footnotesize]
	\node (1) at (0, 0) {};
	\node (2) at (2, 0) {};
	\node (3) at (0, 2) {};
	\node (4) at (1.3, 1.3) {};
	\node (5) at (intersection of 1--2 and 3--4) {};
	\node (6) at (intersection of 1--3 and 2--4) {};
	\draw (2)--(6)--(1)--(5)--(3);
\end{tikzpicture}
\end{equation*}
\caption{A $3$-graph where dots denote vertices and each line denotes an edge (a triple of vertices).}
\label{fig:graph1}
\end{figure}
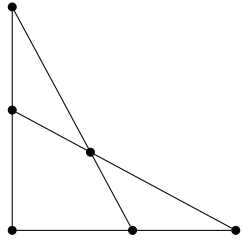

\subsection{Paper Organization} The remainder of the paper is organized as follows. In Section \ref{sec:ldp} we introduce the variational problem,  state the solution to the variational problem for the setting of complete hypergraphs and the 3-graph from \cref{fig:graph1}, and discuss the solution to the variational problem in the case of graphs. In Section \ref{sec:future} we motivate and introduce a conjectural solution to the variational problem in hypergraphs. In Section \ref{sec:lowerbound} we solve the variational problem for complete hypergraphs. In Section \ref{sec:orderrate} we determine the order of the rate function for general hypergraphs. In Section \ref{sec:quad} we solve the variational problem for the 3-graph from \cref{fig:graph1}, and discuss potential future directions. Finally, in Appendix \ref{app:ldp} we show the reduction of the upper tail problem in hypergraphs to the variational problem.

\section{The variational problem}
\label{sec:ldp}

\subsection{Notation}
\label{sec:notation}
It will be convenient to use integral notation. Given a subset $S \subset [k]:=\{1,\dots,k\}$, write $\dx_S := \prod_{i \in S} \dx_i$. Given a finite set $V$ and a function $f \colon V^k \to \RR$, we write
\[
\int f(x_1, x_2, \dots, x_k) \dx_{[k]}:= \frac{1}{|V|^k} \sum_{(x_1, x_2, \dots, x_k) \in V^k} f(x_1, x_2, \dots, x_k).
\]
In other words, we are endowing $V$ with the averaging measure.
For a subset $S \subseteq V^k$, we write
\[
\int_S f(x_1, x_2, \dots, x_k) \dx_{[k]} := \frac{1}{|V|^k} \sum_{(x_1, x_2, \dots, x_k) \in S} f(x_1, x_2, \dots, x_k).
\]

For a vector $(x_1, x_2, \dots, x_k) \in V^k$ and a subset $S \subseteq [k]$, let $x_S := (x_i : i \in S).$ For example, $x_{\{1, 3, 4 \}}$ denotes the $3$-tuple $(x_1, x_3, x_4).$ 

For an $r$-graph $H$ with $V(H) = [k]$, a set $\mathcal{S}$, and a symmetric function $W \colon \mathcal{S}^r \to \RR$ (here \emph{symmetric} means $W(x_1, \dots, x_r) = W(x_{\sigma(1)}, \dots, x_{\sigma(r)})$ for every permutation $\sigma$ of $[r]$), define the $H$-density in $W$ by
\begin{equation}
\label{eq:thwdef}
t(H, W) = \int \prod_{S \in E(H)} W(x_S) \ \dx_{[k]}.
\end{equation}
For example, for $H = K_4^{(3)}$, the complete $3$-graph on $4$ vertices,
\[
 t(K_4^{(3)}, W) = \int W(x_1, x_2, x_3)W(x_1, x_2, x_4)W(x_1, x_3, x_4)W(x_2, x_3, x_4) \dx_1\dx_2\dx_3\dx_4.
\]

We use the following asymptotic notation. Let $f$ and $g$ be nonnegative-valued functions of $n$. As $n \to \infty$, we write $f \ll g$ and $f = o(g)$ to mean $f/g \to 0$; we write $f \ls g$, $f = O(g)$, and $g = \Omega(f)$ to mean $f \le C g$ for some constant $C$; we write $f \asymp g$ and $f = \Theta(g)$ to mean $f \ls g \ls f$; finally, we write $f \sim g$ to mean $f = (1+o(1)) g$.

The \emph{degree} of a vertex of a hypergraph is the number of edges containing the vertex. We write $\Delta(H)$ (or $\Delta$ if $H$ is clear from context) for the maximum degree of $H$.

\subsection{The variational problem}
Let us state the entropic variational problem associated to the large deviation problem for random hypergraphs.

An \emph{edge-weighted $r$-graph} on $n$ vertices with edge-weights in $[0, 1]$ is given by the data $A(G) = (a_{i_1, i_2, \dots, i_r})_{1 \le i_1, i_2, \dots, i_r \le n}$ with entries in $[0,1]$ and is invariant under permuting the order of the indices, i.e., $a_{i_1, i_2, \dots, i_r} = a_{i_{\sigma(1)}, i_{\sigma(2)}, \dots, i_{\sigma(r)}}$ for all permutations $\sigma$ of $[r]$, and also $a_{i_1, i_2, \dots, i_r} = 0$ unless all $i_1, \dots, i_r$ are distinct.
Such an edge-weighted $r$-graph $G$ can be viewed as a symmetric function $G \colon V(G)^r \to [0,1]$. Also we can define $t(H, G)$, the $H$-density in $G$, as in \cref{eq:thwdef}.

Denote the relative entropy of $G$ (relative to the random hypergraph $G_{n,p}^{(k)}$) by
\[
I_p(G) := \sum_{1 \le i_1 < i_2 < \dots < i_r \le n} I_p(a_{i_1, \dots, i_r}) \quad \text{ where } I_p(x) := x \log \frac{x}{p} + (1-x) \log \frac{1-x}{1-p}.
\]
The \emph{variational problem} asks to determine the minimum relative entropy of an edge-weighted $r$-graphs among all those with $H$-density at least $(1+\delta)p^{|E(H)|}$, and we denote its value by
\begin{multline} \label{eq:var}
\phi(H, n, p, \delta) = \inf \Bigl\{ I_p(G) : G \text{ an edge-weighted $r$-graph with edge-weights in $[0,1]$ } \\ \text{ and } t(H, G) \ge (1 + \delta)p^{|E(H)|}\Bigr\},
\end{multline}

It follows from existing theorems on nonlinear large deviations, in particular the work of Eldan \cite{Eld18} (it can also be derived from \cite{CD16} with some work) that the upper tail problem for random hypergraphs reduces to the above variational problem. See the Appendix for details. This theorem is the ``Step (1)'' mentioned in the introduction, and it remains open to improve the range of validity of $p$.

\begin{theorem}[Reduction to variational problem]
\label{thm:fromeldanldp}
Let $H$ be an $r$-graph and let $X_H$ denote the number of copies of $H$ in the random hypergraph $G_{n,p}^{(r)}$, where $p =  p(n)$ satisfies $p > n^{-\frac{1}{6|E(H)|}}\log n$ and $p = o(1)$. Then
\begin{align*}
\PP \paren{X_H \ge (1+\delta) \EE X_H} = \exp\left(-(1+o(1)) \phi(H, n, p, \delta) \right).
\end{align*}
\end{theorem}

In our theorem below, we determine the order of the rate function $\phi(H, n, p, \delta)$.

\begin{theorem}[Order of the solution to the variational problem]
\label{thm:asympbound}
Let $H$ be an $r$-graph with maximum degree $\Delta \ge 2$. If $n^{-1/\D} \ll p \ll 1$, then
\[ \phi(H, n, p, \delta) \asymp n^r p^{\Delta} \log(1/p). \]
\end{theorem}

In \cref{sec:future}, we formulate a conjecture on the missing constant in \cref{thm:asympbound}. We can pin down the leading constant for the hypergraphs described in \cref{thm:rate}: cliques and the 3-graph in \cref{fig:graph1}.

\begin{theorem}[Solution to the variational problem]
\label{thm:varsolutions}
	Fix a hypergraph $H$ and a real $\delta > 0$.
	\begin{enumerate}
		\item [(a)] If $H = K_k^{(r)}$ (clique on $k$ vertices) for $r \ge 3$, then for any $n^{-1/\binom{k-1}{r-1}} \ll p \ll 1$ we have that
		\[ \lim_{n\to\infty}\frac{\phi(H,n,p,\d)}{n^rp^{\binom{k-1}{r-1}}\log(1/p)} = \min\left\{ \frac{\delta^{r/k}}{r!} , \frac{\delta}{ (r-1)! k}  \right\}. \]
		\item[(b)] If $H$ is the 3-graph in \cref{fig:graph1}, then for any $n^{-1/2} \ll p \ll 1$ we have that
		\[ \lim_{n\to\infty}\frac{\phi(H,n,p,\d)}{n^3p^2\log(1/p)} = \frac16\min\left\{ \sqrt{9+3\delta} - 3, \sqrt{\delta} \right\}. \]
	\end{enumerate}
\end{theorem}
\begin{remark}
The recent development of \cite{HMS} (further extended in \cite{BR19}), using combinatorial techniques, reduces the upper tail problem for graphs (i.e., $r=2$), for certain $H$ (but not yet available for all $H$ at the time of this writing), to a \emph{combinatorial variational problem}, which is \cref{eq:var} with edge-weights of $G$ restricted to take values in $\{p,1\}$. For the variational problem for graphs, the asymptotic solutions for the upper tail problem are indeed of this form, and as such, the solution to the restricted variational problem is implied by that of the more general version \cref{eq:var}. In other words, one has asymptotically optimal solutions to the variational problem~\cref{eq:var} coming from planting some subgraph. We conjecture that the same behavior occurs for hypergraphs as well. Nonetheless, it seems that much of the difficulties of solving the combinatorial variational problem remain the same as that of the entropic variational problem \eqref{eq:var}. In this paper, we study the entropic version \cref{eq:var} since it is more general and also because a general (though suboptimal) large deviation principle for random hypergraphs is already available.
\end{remark}

\subsection{Random graphs: solution to the variational problem} \label{sec:graph-solution}

We start by recalling the solution to the large deviation problem for random graphs, which was solved in~\cite{LZ17,BGLZ17}. Fix a graph $H$ with maximum degree $\Delta$. The variational problem \cref{eq:var} for graphs amounts to minimizing $I_p(G) = \sum_{i < j} I_p(a_{ij})$ over all $n$-vertex edge-weighted graphs $G$ (always with edge-weights in $[0,1]$) satisfying $t(H, G) \ge (1+\d)p^{|E(H)|}$. We are interested in the regime when $p \to 0$ (see \cite{LZ15} for discussions in the case of constant $p$), which was solved for $H$ a clique in \cite{LZ17} and for every $H$ in \cite{BGLZ17}.

For connected $H$, the relative entropy $I_p(G)$ is asymptotically minimized by the construction where we plant either a clique or a hub onto the constant $p$ (the constant $p$ corresponds to $G_{n,p}$).
Specifically, \emph{planting a clique} means choosing a parameter $s$ and setting $a_{ij} = 1$ if both $i \le s$ and $j \le s$, and setting $a_{ij} = p$ otherwise. 
This weighted graph corresponds adding a clique on $s$ vertices onto $G_{n,p}$. 
Take $s \sim cp^{\D/2}n$ with some constant $c > 0$, where we assume $1 \ll s \ll n$. Then by considering which vertices of $H$ get mapped to $[s]$ we compute that
\begin{align} t(H, G) &\sim \sum_{S \subseteq V(H)} \left(\frac{s}{n}\right)^{|S|}\left(1-\frac{s}{n}\right)^{|V(H)|-|S|} p^{|E(H)|-|E(H[S])|} \nonumber \\
&\sim \sum_{S \subseteq V(H)} (cp^{\D/2})^{|S|} p^{|E(H)|-|E(H[S])|} \sim
\begin{cases}
(1+c^{|V(H)|})p^{|E(H)|} & \text{ if } H \text{ is } \D\text{-regular,} \\
p^{|E(H)|} & \text{ otherwise.}
\end{cases}
\label{eq:clique-calc}
\end{align}
The last step follows from the fact that for proper non-empty subsets $S$ of $V(H)$, the term in the sum is $o(p^{|E(H)|})$. Indeed, $|E(H[S])| < \D|S|/2$ for any proper nonempty subset $S$ of $V(H)$ as $H$ is connected.

On the other hand, \emph{planting a hub} means
choosing a parameter $s$ and setting $a_{ij} = 1$ if either $i \le s$ or $j \le s$, and setting $a_{ij} = p$ otherwise. 
This weighted graph corresponds to taking $G_{n,p}$ and making some fixed $s$ vertices adjacent to all vertices.
Take $s = \theta p^\D n$ with some constant $\theta > 0$, where we assume $1 \ll s \ll n$. Let $H^\star$ denote the subgraph of $H$ induced by its degree $\Delta$ vertices. 
We compute that
\begin{align} 
t(H, G) 
&\sim \sum_{S \subseteq V(H)} \left(\frac{s}{n}\right)^{|S|}\left(1-\frac{s}{n}\right)^{|V(H)|-|S|} p^{|E[V(H)\bs S]|} \nonumber \\
&\sim \sum_{S \subseteq V(H)}
	(\theta p^\D)^{|S|} p^{|E[V(H)\bs S]|} \sim P_{H^\star}(\theta)p^{|E(H)|},
 \label{eq:hub-calc}
\end{align}
where
\begin{equation} \label{eq:H-star-indep}
P_{H^\star}(\theta) = \sum_{S \text{ independent set of } H^\star} \theta^{|S|}
\end{equation}
is the \emph{independence polynomial} of $H^\star$. In \cref{eq:hub-calc} we have used that $|S| \Delta + |E[V(H)\bs S]|  \ge |E(H)|$ with equality if and only if $S$ is an independent set and all vertices in $S$ have degree $\D$.

The main result of \cite{BGLZ17} (shown earlier in \cite{LZ17} when $H$ is a clique) is that, depending on the range of the parameter $\d$, either planting a clique or planting a hub is asymptotically optimal for connected $H$.

\begin{theorem}[\!\cite{BGLZ17}] \label{thm:BGLZ}
Let $H$ be a connected graph with maximum degree $\Delta \ge 2$. Let $\delta > 0$. Suppose $ p = p(n)$ satisfies $n^{-1/\Delta} \ll p \ll 1$. Then the variational problem \eqref{eq:var} satisfies
\[
\lim_{n \to \infty} \frac{\phi(H, n, p, \delta)}{n^2p^\Delta \log(1/p)} =
\begin{cases}
\min\set{\theta, \tfrac12 \delta^{2/|V(H)|}}, & \text{if $H$ is regular,}
\\
\theta, 	& \text{if $H$ is irregular,}
\end{cases}
\]
where $\theta = \theta(H, \delta)$ is the unique positive solution to $P_{H^\star}(\theta) = 1+\delta$.
\end{theorem}

For disconnected $H$, the asymptotically optimal solution comes from simultaneously planting a clique and a hub \cite[Section 7]{BGLZ17}.

\section{The conjectural solution to the hypergraph variational problem}
\label{sec:future}

\subsection{Extending the graph solution: cliques and hubs}

Let us attempt to formulate a conjectural solution to the variational problem \cref{eq:var} for a general hypergraph $H$, similar to \cref{thm:BGLZ} for graphs.

Similar to the graph case described in \cref{sec:graph-solution}, a natural guess would be that the variational problem is asymptotically solved by planting cliques and hubs.

Given an $r$-graph $G_0$ with vertex set $[n]$, we say that an edge-weighted $r$-graph $G$ on vertex set $[n]$ arises from \emph{planting $G_0$} if the edge-weights $(a_e)_{e\in E(G)}$ of $G$ satisfy $a_e = 1$ whenever $e$ is an edge of $G_0$ and $a_e = p$ otherwise.

Let us attempt to state the asymptotic solution to the variational problem. We will describe a parameterized family of edge-weighted graphs that could serve as the asymptotic optimizer. We always assume a fixed hypergraph $H$ and $n^{-c_H} \ll p \ll 1$.

Motivated by the graph case, where the solution comes from planting a union of a clique and a hub (when $H$ is connected, one plants either a clique or a hub), one could conjecture that the same happens in the hypergraph setting as well. For hypergraphs, a \emph{clique} consists of all $r$-tuples contained in a given set $S$ of vertices, whereas a \emph{hub} consists of all $r$-tuples that intersect a given set $S$ of vertices. In particular, we shall only consider planting cliques and hubs where the corresponding set $S$ is a prefix of the vertex set $[n]$ (though the size of $S$ could be different for the clique and the hub even if both are simultaneously planted).

\begin{nconj}
\label{nconj1}
The variational problem \cref{eq:var} is asymptotically optimized by planting a clique and a hub.
\end{nconj}

The above na{\"i}ve conjecture is true for graphs. For hypergraphs, we will also show that it is true when $H$ is a clique. However, the na{\"i}ve conjecture fails in general.

One attempt to rectify the conjecture is to extend the notion of cliques and hubs for hypergraphs.
In an $r$-graph, define a \emph{$k$-hub} to be all $r$-tuples of vertices that contains at least $k$ vertices from some specified prefix of the vertex set $[n]$.

For example, in an $r$-graph, $1$-hubs are hubs and $r$-hubs are cliques. 

\begin{nconj}
\label{nconj2}
The variational problem \cref{eq:var} for $r$-graphs is asymptotically optimized by planting a union of a $1$-hub, a $2$-hub, \dots, and a $r$-hub.
\end{nconj}

Unfortunately, \cref{nconj2} remains false, as we now give a counterexample. 
The counterexample 3-graph $H$ has 13 vertices and 15 edges. Its edges are given as follows (also see \cref{fig:counterexample-H}), where whenever we write a pair, we extend it to a triple by adding a new dummy vertex (all the dummy vertices are distinct):
\begin{align*}
 &(A, B), (B, C), (A, C), (A, D, F), (B, D, E), (C, E, F), \\ 
 &(A', B'), (B', C'), (A', C'), (A', D', F'), (B', D', E'), (C', E', F'), \\ 
 &(D, D', G), (E, E', G), (F, F', G). 
\end{align*}

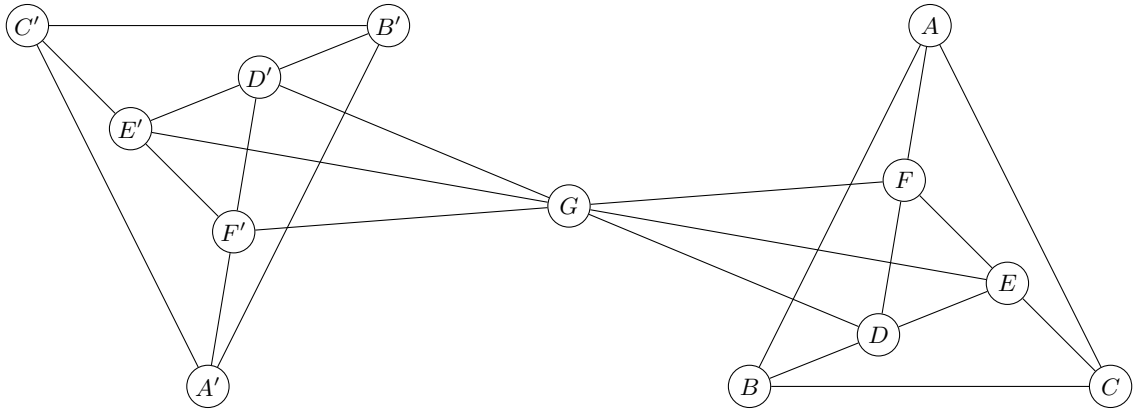
\begin{figure}
\begin{center}
\begin{tikzpicture}[baseline=(current bounding box.center),scale=1.2, every node/.style={draw, circle, fill=white, inner sep = 1pt, minimum width = 16pt},font=\footnotesize]
	\node (1) at (0, 0) {$G$};
	\node (2) at (4, 2) {$A$};
	\node (3) at (2, -2) {$B$};
	\node (4) at (6, -2) {$C$};
	\node (5) at (26/7, 2/7) {$F$};
	\node (6) at (24/7, -10/7) {$D$};
	\node (7) at (34/7, -6/7) {$E$};
	\node (8) at (-4, -2) {$A'$};
	\node (9) at (-2, 2) {$B'$};
	\node (10) at (-6, 2) {$C'$};
	\node (11) at (-26/7, -2/7) {$F'$};
	\node (12) at (-24/7, 10/7) {$D'$};
	\node (13) at (-34/7, 6/7) {$E'$};
	\draw (2)--(3)--(4)--(2);
	\draw (8)--(9)--(10)--(8);
	\draw (5)--(1)--(11);
	\draw (6)--(1)--(12);
	\draw (7)--(1)--(13);
	\draw (2)--(5)--(6);
	\draw (3)--(6)--(7);
	\draw (4)--(7)--(5);
	\draw (8)--(11)--(12);
	\draw (9)--(12)--(13);
	\draw (10)--(13)--(11);
\end{tikzpicture}
\end{center}
\caption{Counterexample graph to Naive Conjecture \ref{nconj2}. Only vertices of degree $3$ are drawn -- we can view the structure as a $3$-graph by completing each edge drawn with only two vertices with a new dummy vertex (a distinct dummy vertex for each such edge).}
\label{fig:counterexample-H}
\centering
\end{figure}

Note that all the labeled vertices (i.e., other than the omitted dummy vertices) have degree equal to the maximum degree $\Delta = 3$.

Consider the 3-graph $G_0$ on vertex set $[n]$ whose edges are all $\{i_1, i_2, i_3\} \in \binom{[n]}{3}$ with $i_1 < i_2 < i_3$ with $i_1 \le c p^{3/2}n$ and $i_2,i_3 \le c' p^{3/4} n$ for some appropriately chosen constants $c,c' > 0$. This construction is not a union of $k$-hubs. However, as we shall verify in \cref{sec:counterexample-proof}, this construction performs better than a union of $k$-hubs.

\subsection{Mixed hubs}

The above counterexample construction motivates the following generalization of a $k$-hub. In an $r$-graph, we define a \emph{mixed hub} to be the subgraph consisting of all edges $\{i_1, \dots, i_r\}$ with $i_1 \le n_1$, \dots, $i_r \le n_r$ for some specified $n_1, \dots, n_r$. Note that this set of edges is almost but not quite the same as the cartesian box $[n_1] \times \cdots \times [n_r]$ since edges consist of $r$-tuples of distinct vertices. Note that all $k$-hubs can be described in this form, by setting $k$ of the $n_i$ to be equal, and setting the rest to $n$.

Naturally, one could conjecture that the optimal solution consists of a union of mixed hubs. Though such a statement would not make a particularly useful conjecture since the space of possibilities for a union of mixed hubs is not finitely parametrizable. In order to make a more useful conjecture, let us look at what kind of widths $n_i$'s we should take in the construction of the mixed hubs in order to contribute meaningfully to the variational problem.

Going forward, we consider constructions with a given sequence $p = p(n) \gg n^{-1/\Delta}$ and $p = o(1)$. 

Let $t_1, \dots, t_r \ge 0$ with $t_1 + \cdots  + t_r = 1$. In an $r$-graph, a \emph{$(t_1, t_2, \dots, t_r)$-mixed hub} consists of all edges that can be written as an element of $[n_1] \times \cdots \times [n_r]$ with $n_i \sim c_i p^{t_i \Delta} n$, where $c_1, \dots, c_r > 0$ are constants, and $n_i = n$ whenever $t_i = 0$.

The requirement $t_1 + \cdots + t_r = 1$ in a mixed hub is so that the construction achieves the correct order of magnitude in the variational problem \cref{eq:var}, namely so that an edge-weighted $r$-graph obtained by planting a mixed hub has $I_p(G) = \Theta(n^k p^\Delta \log(1/p))$; c.f.\ \cref{thm:asympbound}.

Our construction will consist of taking a union of mixed hubs, possibly with different $(t_1, \dots, t_r)$ parameters. We say that a finite collection of mixed hubs is \emph{compatible} if there exists some function $c \colon [0,1] \to \RR_{\ge 0}$ with $c(0) = 1$ such that for each $(t_1, \dots, t_r)$-mixed hub in the collection, the corresponding constants $(c_1, \dots, c_r)$ satisfy $c_i = c(t_i)$ for all $i$ (the same function $c$ is used for all mixed hubs in the collection). Compatibility is necessary to ensure that various mixed hubs can be planted and contribute towards the same labeling of $H$ (defined in Definition \ref{def:gammah1}).

In other words, a compatible collection of mixed hubs is indexed by a finite set $\S$ (called the \emph{indexing set}) along with a function $c \colon [0,1] \to \RR_{\ge 0}$. The elements of the indexing set $\S$ are ordered tuples $(t_1, \dots, t_r)$ each satisfying $t_1, \dots,  t_r \ge 0$ and $t_1 + \cdots + t_r = 1$, and $\S$ is invariant under permutations of coordinates, i.e., $(t_{\sigma(1)}, \dots, t_{\sigma(r)}) \in \S$ whenever $(t_1, \dots, t_r) \in \S$ and $\sigma$ is a permutation of $[r]$. Here the only relevant values of $c(\cdot)$ are $c(t_i)$ for some $t_i$ appearing in as a coordinate in some element of $\S$. We may as well set $c(t) = 0$ unless $t$ appears as a coordinate in some element of $\S$.

For example, in a 3-graph, a $(1,0,0)$-mixed hub is a 1-hub, a $(1/2, 1/2, 0)$-mixed is a 2-hub, and a $(1/3,1/3,1/3)$-mixed hub is a 3-hub. The counterexample given above to \cref{nconj2} involves a $(1/4,1/4,1/2)$-hub.

We define the \emph{volume} of a compatible collection of mixed hubs indexed by $(\S,c)$ to be
\begin{equation}\label{eq:vol} \Vol(\S, c) = \sum_{(t_1, \dots, t_r) \in \S} \prod_{i=1}^r c(t_i). \end{equation}
The number of edges in the union of these mixed hubs is $(1+o(1)) \Vol(\S,c) p^\D n^r / r!$.

\medskip

Let us now compute $t(H, G)$ where $G$ is obtained by planting a compatible collection of mixed hubs indexed by $(\S,c)$. Let $V(G) = [n]$.

Now we describe how to estimate $t(H, G)$ by extending the calculations \cref{eq:clique-calc} and \cref{eq:hub-calc}. We can partition the vertex set $V(G) = [n]$ based on the largest value of $t$ that appears as a coordinate of $\S$ (also allowing $t = 0$) such that the vertex $i \in V(G)$ has $i \le c(t) p^{t \Delta}$, where we are crucially using the assumption of compatibility. And then we partition the set of all copies of $H$ in $G$ based on which part in the partition each vertex of $H$ gets mapped to. We can enumerate the partition induced on the set of copies of $H$ using a function $f \colon V(H) \to [0,1]$, where $f$ takes on values that are either 0 or some number that appears as a coordinate in $\S$. 

Define $E_f$ to be the set of edges $\{i_1, i_2, \dots, i_r\}$ in $E(H)$ such that there is a $(f(i_1), f(i_2), \dots, f(i_r))$-mixed hub in the construction of $G$. The contributions to $t(H, G)$ indexed by $f$ (i.e., corresponding to homomorphisms from $H$ to $G$ where each $v \in V(H)$ is mapped to some $i \in V(G)$ with $i \le c(f(v)) p^{f(v)\Delta}$) is then 
\begin{equation} \label{eq:smallcontrib} 
\sim p^{|E(H)|-|E_f|} \prod_{v \in V(H)} c(f(v)) p^{f(v)\D}, \end{equation} where the first factor comes from the fact that all edges in $E_f$ are mapped to edges in $G$ with weight $1$, while the other edges in $E(H)$ are mapped to edges in $G$ with weight $p$, and the second factor comes from the number of choices for the image of each $v \in V(H)$. We can check that the contribution \cref{eq:smallcontrib} is $o(p^{|E(H)|})$ unless:
\begin{itemize}
\item for every edge $e \in E(H)$, we have $\sum_{v \in e} f(v) = 0$ or $1$, and
\item for all vertices $v \in V(H)$ with $\deg(v) < \D$, we have $f(v) = 0.$
\end{itemize}
Indeed, note that because $t_1 + \cdots + t_r = 1$ for every $(t_1, \dots, t_r) \in S$, we have that \[ |E_f| = \sum_{e \in E_f} \sum_{v \in e} f(v) = \sum_{v \in V(H)} f(v) |\{e \in E_f : v \in e \}| \le \sum_{v \in V(H)} f(v)\D. \] 
We need equality to hold so that \cref{eq:smallcontrib} is on the order of  $p^{|V(H)|}$, and equality above is equivalent to $\sum_{v \in e} f(v) = 1$ for all $e \in E_f$, and that $f(v) = 0$ if $\deg v < \D$.

Let us introduce some notation to make precise the set of vertex-labelings $f$ of $H$ that can come up in \cref{eq:smallcontrib}. We first define a set $\wt\Gamma_H$ below but we will need to further restrict them later.
\begin{definition}[Labelings $\wt\Gamma_H$]
\label{def:gammah1}
Let $H$ be an $r$-graph. Define $\wt\Gamma_H$ to be the set of functions $f:V(H) \to [0, 1]$ satisfying
\begin{itemize}
\item $f(v) = 0$ for all $v \in V(H)$ with $\deg(v) < \Delta$, and
\item $\sum_{v \in e} f(v) \in \{0,1\}$ for each $e \in E(H)$.
\end{itemize}
Define $\wt\Gamma_H(\S)$ to be the set of functions $f \in \wt\Gamma_H$ such that $(f(v))_{v \in e} \in \S \cup \{(0,0,\dots,0)\}$ for all $e \in E(H)$.
\end{definition}

\subsection{Stable labelings}
An issue with the above computation is the set of labelings $\wt\Gamma_H$ may be infinite, so that looking at this class of constructions would not produce a conjectured value of the variational that is given by a finitely parametrizable optimizable problem. For example, consider the set $\wt\Gamma_H$ where $H$ is a regular bipartite graph with bipartition $H = A \cup B.$ Then, for each $t \in [0,1]$, the labeling $f(v) = t$ for all $v \in A$, $f(v) = 1-t$ for all $v \in B$ is in $\wt\Gamma_H$. To limit the set of vertex labelings, we define a set of \emph{stable labelings}.
\begin{definition}[Stable labelings $\Gamma_H$]
\label{def:gammah}
Let $H$ be an $r$-graph. We call $f \in \wt\Gamma_H$ \emph{stable} if there does not exist a different $f' \in \wt\Gamma_H$ such that 
\begin{itemize}
\item $\sum_{v \in e} f(v) = \sum_{v \in e} f'(v)$ for all $e \in E(H)$, and
\item for all distinct $v_1, v_2 \in V(H)$, one has $f(v_1) = f(v_2)$ if and only if $f'(v_1) = f'(v_2)$, and
\item for all $v \in V(H)$, one has $f(v) = 0$ if and only if $f'(v) = 0$.
\end{itemize}
Denote the set of stable labelings by $\Gamma_H$. 
Define $\Gamma_H(\S)$ to be the set of functions $f \in \Gamma_H$ such that $(f(v))_{v \in e} \in \S \cup \{(0,0,\dots,0)\}$ for all $e \in E(H)$.
\end{definition}

In other words, $f \in \wt\Gamma_H$ if it arises as a solution to a system of linear equations given in \cref{def:gammah1} (where we make a choice of $\{0,1\}$ for each $e \in E(H)$). However, this system may have more than one solution, in which case we continue to add constraints of the form $f(v_1) = f(v_2)$ for some pair of distinct $v_1,v_2\in V(H)$ or $f(v) = 0$ for some $v \in V(H)$ until the system is forced to have a unique solution. Since there are a finite number of such systems of linear equations for each $H$, the set $\Gamma_H$ of stable labelings is finite. 
Examples of computation with stable labelings are given below in \cref{sec:labeling-examples}.

The intuition for why stability may be required is that if some construction gives rise to a labeling $f$ that is not stable, then perhaps, due to some convexity-like reasons, one can can perturb the construction so that some labels become equal.

It suffices to restrict to $\S$ that \emph{respects stable labelings}, meaning that every $(t_1, \dots, t_r) \in \S$ appears as the labels of some edge in some $f \in \Gamma_H$. Then there are only finitely many possibilities for $\S$ for each $H$.

Our conjecture, stated informally, is that the asymptotically optimal solutions to the variational problem \cref{eq:var} arise from planting a compatible collection of mixed hubs that respect stable labelings.

A more formal version is stated below. Given a compatible collection of mixed hubs indexed by set $\S$ and function $c \colon [0,1] \to \RR_{\ge 0}$, define
\[ P_H(\S, c) = \sum_{f \in \Gamma_H(\S)} \prod_{v\in V(H)} c(f(v)). \]
We write
\begin{align}
	\label{eq:rho}
	\rho_H(\d) := \inf\{\Vol(\S, c) :\ & \text{compatible collection of mixed hubs that respect stable labelings } \nonumber \\ 
	&\text{indexed by $(\S,c)$ with } P_H(\S, c) \ge 1+\d\}.
\end{align}

\begin{conj}
\label{conj:main}
Fix an $r$-graph $H$ with maximum degree $\D$. For $n \to \infty$ and $p := p(n)$ satisfying $n^{-1/\D} \ll p \ll 1$, we have that \[ \min\{I_p(G_n) : G_n \in \G_n, t(H, G_n) \ge (1+\d)p^{|E(H)|} \} \sim \frac{1}{r!}n^r p^\D \rho_H(\d) \log(1/p). \]
\end{conj}
A routine computation similar to \cref{eq:clique-calc} and \cref{eq:hub-calc} shows that, with $G$ being the edge-weighted $r$-graph obtained by planting the compatible collection of mixed hubs indexed by $(\S, c)$ on top of the constant $p$,
as long as $n^{-1/\D} \ll p \ll 1$,
\[ 
t(H, G) \ge (1-o(1))P_H(\S, c)p^{|E(H)|},
\]
where we have crucially used the assumption of compatibility.
Also,
\[
I_p(G) \sim \frac{1}{r!}n^r p^\D \Vol(\S, c) \log(1/p). 
\] 
This shows the upper bound to \cref{conj:main}.

\begin{lemma}[Upper bound to \cref{conj:main}]
\label{lemma:conjupper}
Fix an $r$-graph $H$ with maximum degree $\D$. For $n \to \infty$ and $p := p(n)$ satisfying $n^{-1/\D} \ll p \ll 1$, we have that \[ \min\{I_p(G_n) : G_n \in \G_n, t(H, G_n) \ge (1+\d)p^{|E(H)|} \} \le (1+o(1))\frac{1}{r!}n^r p^\D \rho_H(\d) \log(1/p). \]
\end{lemma}

\subsection{Examples} \label{sec:labeling-examples}
We now give a number of examples for \cref{conj:main}.
We write $c_x := c(x)$ to make the formulas more readable.
\begin{example}
\label{ex:cliques}
We start by explaining how the conjecture as stated above applies to the case where $H = K_k^{(r)}.$ A direct calculation shows that for $k > r$, the only labelings $f \in \wt\Gamma_H$ (\cref{def:gammah1}) are
\begin{itemize}
\item $f(v) = 1/r$ for all $v \in V(H)$
\item $f(v) = 1$ for some $v$, and $f(w) = 0$ for all $w \in V(H)$ with $w \neq v.$
\end{itemize}
Therefore, the set of stable labelings $\Gamma_H$ is the same as $\wt\Gamma_H$. Here, we assume that the indexing set $\S$ contains the tuples $(1,0,\cdots,0)$, $(1/r,1/r,\cdots,1/r)$ and their permutations, corresponding to $1$-hubs and $r$-hubs. We can assume this because we can set $c_1 = 0$ or $c_{1/r} = 0$ to handle the other cases.

For a choice of function $c$, we can compute that $P_H(\S, c) = 1 + c_{1/r}^k + k c_1$ and $\Vol(\S, c) = c_{1/r}^r + r c_1$.  Therefore, we get that
\[ \rho_H(\delta) 
= \inf\set{ c_{1/r}^r + r c_1 : c_{1/r}^k + k c_1 \ge \delta, \ c_{1/r}, c_1 \in \RR_{\ge 0}} 
= \min \left\{ \delta^{r/k} , \frac{r \delta}{k}  \right\}, \] which matches the result in \cref{thm:rate}(a).
\end{example}

\begin{example}
\label{ex:2graph}
We now explain how the conjecture applies to 2-graphs (\cref{thm:BGLZ}). 
Let $\S$ be any symmetric combination of mixed hubs compatible with $c$. For simplicity, let $H$ be a connected 2-graph. Let $V(H)$ denote the vertex set of $H$. One can check that the stable labelings $f \in \Gamma_H$ are all of the following forms:
\begin{itemize}
\item $f(v) = 1$ for all $v$ in some independent set $I$ of $H$ such that $\deg(v) = \Delta$ for all $v \in I$ and $f(v) = 0$ for $v \in V(H)\backslash I$, or
\item If $H$ is regular, $f(v) = \frac{1}{2}$ for all $v \in V(H)$.
\end{itemize}
Note that there are many labelings $f \in \wt\Gamma_H$ which are not stable in the case where $H$ is regular with bipartition $V(H) = A \cup B$, namely labelings $f$ with $f(v) = x$ for $v \in A$ and $f(v) = 1-x$ for $v \in B.$

Here, we assume that the indexing set $\S = \{(1,0), (0,1), (1/2,1/2)\}$, corresponding to $1$-hubs and $2$-hubs. We can assume that $\S$ contains all these triples  because we can set $c_1 = 0$ or $c_{1/2} = 0$ to handle the other cases.

For a function $c: [0,1] \to \R_{\ge0}$, if $H$ is irregular, it is easy to check that $P_H(\S, c) = P_{H^\star}(c_1)$, where the latter $P_{H^\star}$ is the independence polynomial as in \cref{eq:H-star-indep}. If $H$ is regular, then we can compute that $P_H(\S, c) = P_{H^\star}(c_1) + c_{1/2}^{|V(H)|}.$ In all cases, we have that $\Vol(\S, c) = 2c_1 + c_{1/2}^2$. It follows that
\[
\rho_H(\d)  =
\begin{cases}
	\min\{2\theta, \delta^{2/|V(H)|}\} & \text{if $H$ is regular,} \\
	2\theta & \text{if $H$ is irregular.}
\end{cases}
\]
This matches the result in \cref{thm:BGLZ}.
\end{example}

\begin{example}
\label{ex:quad}
Let $H$ be the 3-graph from \cref{fig:graph1}, reproduced below.
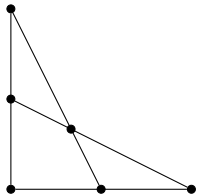
\begin{center}
\begin{tikzpicture}[scale=.4, every node/.style={draw, circle, black, fill, inner sep = 0pt, minimum width = 3pt},font=\footnotesize]
	\node (1) at (0, 0) {};
	\node (2) at (3, 0) {};
	\node (3) at (0, 3) {};
	\node (4) at (2, 2) {};
	\node (5) at (intersection of 1--2 and 3--4) {};
	\node (6) at (intersection of 1--3 and 2--4) {};
	\draw (2)--(6)--(1)--(5)--(3);
\end{tikzpicture}
\end{center}
Then $\wt\Gamma_H$ (see \cref{def:gammah1}) consists of the following assignments, where the label of a vertex $v$ denotes $f(v).$ We have omitted $0$ labels, and we have denoted the number of automorphisms each labeling has.
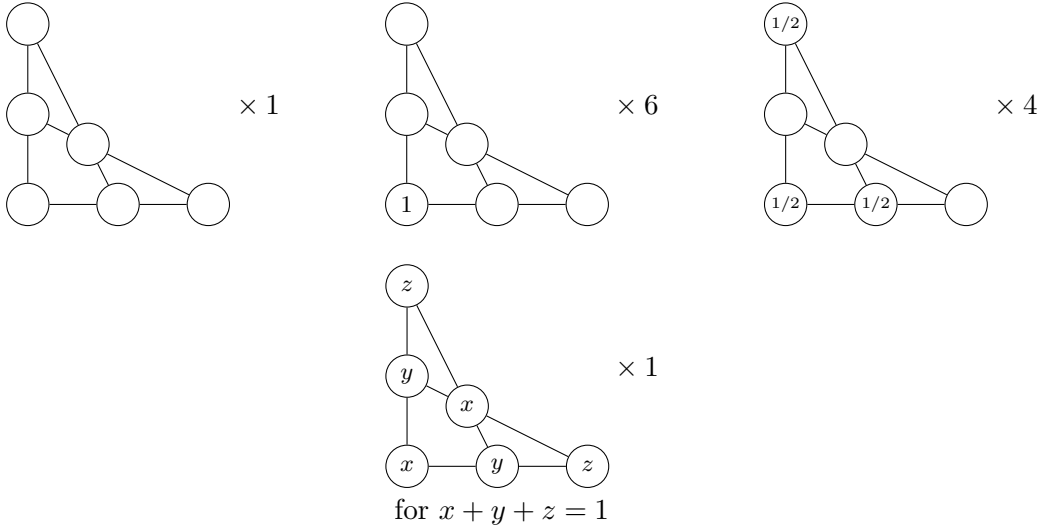
\begin{align*} 
\begin{tikzpicture}[baseline=(current bounding box.center),scale=.4, every node/.style={draw, circle, fill=white, inner sep = 1pt, minimum width = 16pt},font=\footnotesize]
	\node (1) at (0, 0) {};
	\node (2) at (3, 0) {};
	\node (3) at (0, 3) {};
	\node (4) at (2, 2) {};
	\node (5) at (intersection of 1--2 and 3--4) {};
	\node (6) at (intersection of 1--3 and 2--4) {};
	\draw (1)--(2)--(5)--(4)--(3)--(1);
	\draw (3)--(6)--(4)--(2);
\end{tikzpicture} &\times 1
&
\begin{tikzpicture}[baseline=(current bounding box.center),scale=.4, every node/.style={draw, circle, fill=white, inner sep = 1pt, minimum width = 16pt},font=\footnotesize]
	\node (1) at (0, 0) {1};
	\node (2) at (3, 0) {};
	\node (3) at (0, 3) {};
	\node (4) at (2, 2) {};
	\node (5) at (intersection of 1--2 and 3--4) {};
	\node (6) at (intersection of 1--3 and 2--4) {};
	\draw (1)--(2)--(5)--(4)--(3)--(1);
	\draw (3)--(6)--(4)--(2);
\end{tikzpicture} &\times 6 
&
\begin{tikzpicture}[baseline=(current bounding box.center),scale=.4, every node/.style={draw, circle, fill=white, inner sep = 1pt, minimum width = 16pt},font=\footnotesize]
	\node (1) at (0, 0) {\Tiny{1/2}};
	\node (2) at (3, 0) {\Tiny{1/2}};
	\node (3) at (0, 3) {};
	\node (4) at (2, 2) {};
	\node (5) at (intersection of 1--2 and 3--4) {};
	\node (6) at (intersection of 1--3 and 2--4) {\Tiny{1/2}};
	\draw (1)--(2)--(5)--(4)--(3)--(1);
	\draw (3)--(6)--(4)--(2);
\end{tikzpicture} &\times 4
\end{align*}
\begin{align*}
\begin{tikzpicture}[baseline=(current bounding box.center),scale=.4, every node/.style={draw, circle, fill=white, inner sep = 1pt, minimum width = 16pt},font=\footnotesize]
	\node (1) at (0, 0) {$x$};
	\node (2) at (3, 0) {$y$};
	\node (3) at (0, 3) {$y$};
	\node (4) at (2, 2) {$x$};
	\node (5) at (intersection of 1--2 and 3--4) {$z$};
	\node (6) at (intersection of 1--3 and 2--4) {$z$};
	\draw (1)--(2)--(5)--(4)--(3)--(1);
	\draw (3)--(6)--(4)--(2);
\end{tikzpicture} &\times 1 \\
\text{ for } x+y+z=1 
\end{align*}
The first three labelings all correspond to stable labelings. For the last set of labels in $\wt\Gamma_H$, note that the ways to make it stable are to
\begin{itemize}
\item $x = y = 0$ and $z = 1$ (and symmetric versions), or
\item $x = 0$, $y = z$ (and symmetric versions), or
\item $x = y = z$.
\end{itemize}
This gives rise to the following stable labelings.
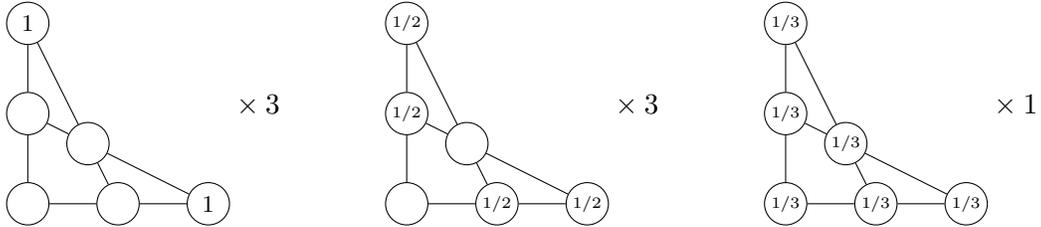
\begin{align*}
\begin{tikzpicture}[baseline=(current bounding box.center),scale=.4, every node/.style={draw, circle, fill=white, inner sep = 1pt, minimum width = 16pt},font=\footnotesize]
	\node (1) at (0, 0) {};
	\node (2) at (3, 0) {};
	\node (3) at (0, 3) {};
	\node (4) at (2, 2) {};
	\node (5) at (intersection of 1--2 and 3--4) {1};
	\node (6) at (intersection of 1--3 and 2--4) {1};
	\draw (1)--(2)--(5)--(4)--(3)--(1);
	\draw (3)--(6)--(4)--(2);
\end{tikzpicture} &\times 3
&
\begin{tikzpicture}[baseline=(current bounding box.center),scale=.4, every node/.style={draw, circle, fill=white, inner sep = 1pt, minimum width = 16pt},font=\footnotesize]
	\node (1) at (0, 0) {};
	\node (2) at (3, 0) {\Tiny{1/2}};
	\node (3) at (0, 3) {\Tiny{1/2}};
	\node (4) at (2, 2) {};
	\node (5) at (intersection of 1--2 and 3--4) {\Tiny{1/2}};
	\node (6) at (intersection of 1--3 and 2--4) {\Tiny{1/2}};
	\draw (1)--(2)--(5)--(4)--(3)--(1);
	\draw (3)--(6)--(4)--(2);
\end{tikzpicture} &\times 3
&
\begin{tikzpicture}[baseline=(current bounding box.center),scale=.4, every node/.style={draw, circle, fill=white, inner sep = 1pt, minimum width = 16pt},font=\footnotesize]
	\node (1) at (0, 0) {\Tiny{1/3}};
	\node (2) at (3, 0) {\Tiny{1/3}};
	\node (3) at (0, 3) {\Tiny{1/3}};
	\node (4) at (2, 2) {\Tiny{1/3}};
	\node (5) at (intersection of 1--2 and 3--4) {\Tiny{1/3}};
	\node (6) at (intersection of 1--3 and 2--4) {\Tiny{1/3}};
	\draw (1)--(2)--(5)--(4)--(3)--(1);
	\draw (3)--(6)--(4)--(2);
\end{tikzpicture} &\times 1
\end{align*}

We can assume that the indexing set $\S$ contains all the tuples $(1,0,0)$, $(1/2,1/2,0)$, $(1/3,1/3,1/3)$ and their permutations, corresponding to $1$-hubs, $2$-hubs, and $3$-hubs, since we can set $c_1 = 0$, $c_{1/2} = 0$, or $c_{1/3} = 0$ to handle the other cases. We also assume that $c$ is zero outside $\{0, 1/3, 1/2, 1\}$.

By the above construction, we have that 
\begin{equation} \label{eq:plant13}
P_H(\S, c) \ge 1 + 6c_1 + 3c_1^2 + 4c_{1/2}^3 + 3c_{1/2}^4 + c_{1/3}^6 
\end{equation}
and
\[ 
\Vol(\S, c) = 3c_1 + 3c_{1/2}^2 + c_{1/3}^3.
\] 
Therefore, we get that 
\begin{align*}
\rho_H(\d) 
&= \inf \{ 3c_1 + 3c_{1/2}^2 + c_{1/3}^3 : 6c_1 + 3c_1^2 + 4c_{1/2}^3 + 3c_{1/2}^4 + c_{1/3}^6 \ge \d,\  c_{1/2}, c_{1/3}, c_1 \in \RR_{\ge 0} \} 
\\&= \min\left( \sqrt{9+3\d}-3, \sqrt{\d} \right),
\end{align*}
achieved by the triples \[ (c_1,c_{1/2},c_{1/3}) = \left(\frac{\sqrt{9+3\d}-3}{3}, 0, 0\right) \enspace \text{ and } \enspace (c_1,c_{1/2},c_{1/3}) = \left(0, 0, \d^\frac16\right) \]
respectively. By \cref{lemma:conjupper}, we have obtained the upper bound in \cref{thm:varsolutions}(b).
\end{example}

\subsection{Proof that \cref{nconj2} is false} \label{sec:counterexample-proof}
In this section we show that \cref{nconj2} is false for the example $3$-graph $H$ given in \cref{fig:counterexample-H}. Note that in our depiction of $H$ in \cref{fig:counterexample-H}, some of the edges only have $2$ vertices. This $3$-graph was chosen so that each vertex that was drawn has degree equal to $\D = 3$ and that there is a unique stable labeling $f: V(H) \to [0, 1]$ satisfying $\sum_{v \in e} f(v) = 1$ for all edges $e$ in the 3-graph. This labeling is shown in \cref{fig:label}. Additionally, this labeling $f$ does not have image in the set $\left\{ 0, \frac13, \frac12, 1 \right\}$, which are the labels corresponding to $k$-hubs for $1 \le k \le 3$.

\begin{figure}
\begin{center}
\begin{tikzpicture}[baseline=(current bounding box.center),scale=1.2, every node/.style={draw, circle, fill=white, inner sep = 1pt, minimum width = 16pt},font=\footnotesize]
	\node (1) at (0, 0) {\Tiny{1/2}};
	\node (2) at (4, 2) {\Tiny{1/2}};
	\node (3) at (2, -2) {\Tiny{1/2}};
	\node (4) at (6, -2) {\Tiny{1/2}};
	\node (5) at (26/7, 2/7) {\Tiny{1/4}};
	\node (6) at (24/7, -10/7) {\Tiny{1/4}};
	\node (7) at (34/7, -6/7) {\Tiny{1/4}};
	\node (8) at (-4, -2) {\Tiny{1/2}};
	\node (9) at (-2, 2) {\Tiny{1/2}};
	\node (10) at (-6, 2) {\Tiny{1/2}};
	\node (11) at (-26/7, -2/7) {\Tiny{1/4}};
	\node (12) at (-24/7, 10/7) {\Tiny{1/4}};
	\node (13) at (-34/7, 6/7) {\Tiny{1/4}};
	\draw (2)--(3)--(4)--(2);
	\draw (8)--(9)--(10)--(8);
	\draw (5)--(1)--(11);
	\draw (6)--(1)--(12);
	\draw (7)--(1)--(13);
	\draw (2)--(5)--(6);
	\draw (3)--(6)--(7);
	\draw (4)--(7)--(5);
	\draw (8)--(11)--(12);
	\draw (9)--(12)--(13);
	\draw (10)--(13)--(11);
\end{tikzpicture}
\end{center}
\caption{The unique labeling $f$ satisfying $\sum_{v\in e} f(v) = 1$ for all edges $e$ in $3$-graph $H$ in \cref{fig:counterexample-H}.}
\label{fig:label}
\end{figure}
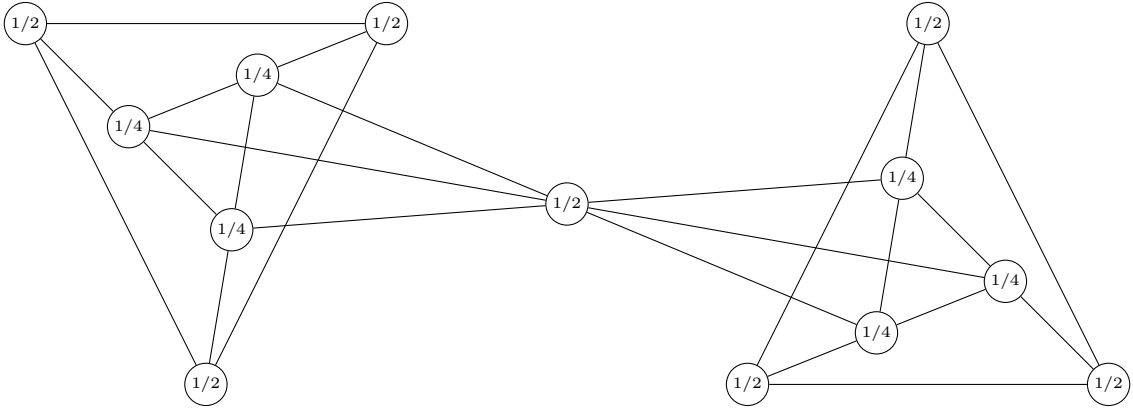

\begin{prop}
\label{prop:counterupper}
Let $H$ be the 3-graph in \cref{fig:counterexample-H}. We have $\rho_H(\delta) \le 6\delta^{1/5}$.
\end{prop}
\begin{proof}
Consider the labeling $f \in \Gamma_H$ in \cref{fig:label}. Let $\S$ consist of the $3$-tuples $(1,0,0)$, $(1/2,1/2,0)$, $(1/4,1/4,1/2)$ and their permutations. The triple $(1/4,1/4,1/2)$ does not correspond to any $k$-hub. Set $c_0 = 1, c_{1/2} = \delta^{1/10}$, and $c_{1/4} = \delta^{1/20}$. Let $f$ be the labeling in \cref{fig:label}. This gives us that for \[ P_H(\S, c) \ge 1 + \prod_{v\in V(H)}c_{f(v)} \ge 1+c_{1/2}^7c_{1/4}^6 = 1+\d. \] Also, we have that \[ \Vol(\S, c) = 3c_{1/2}^2 + 3c_{1/2}c_{1/4}^2 = 6\delta^{1/5}. \]
Thus $\rho_H(\delta) \le 6\delta^{1/5}$ as claimed.
\end{proof}
We now show that compatible collections of $1$-hubs, $2$-hubs, and $3$-hubs cannot achieve the bound in \cref{prop:counterupper}.

\begin{prop}
\label{prop:counterlower}
For the $3$-graph $H$ in \cref{fig:counterexample-H} and sufficiently large constant $\delta$, we have that for every compatible collection of $k$-hubs indexed by $\S$ and $c$ with $P_H(\S, c) \ge 1+\d$, we have that $\Vol(\S, c) = \Omega(\d^{3/14})$.
\end{prop}

\begin{proof}
We can assume that the indexing set $\S$ contains all the tuples $(1,0,0)$, $(1/2,1/2,0)$, \\ $(1/3,1/3,1/3)$ and their permutations, corresponding to $1$-hubs, $2$-hubs, and $3$-hubs, since we can handle the other cases by setting some of $c_1,c_{1/2},c_{1/3}$ to equal zero.

Since $P_H(\S, c) \ge 1 + \delta$, there exists some nonzero $f \in \Gamma_H(\S)$ such that $\prod_{v \in V(H)} c_{f(v)} \ge \delta/|\Gamma_H(\S)|$. Every $f \in \Gamma_H(\S)$ has 
\begin{equation} \label{eq:143}
 \sum_{v \in V(H)} f(v) = \frac13 \sum_e \left(\sum_{v \in e} f(v) \right) \le \frac{14}{3},
\end{equation}
since $\sum_{v \in e} f(v) \in \{0,1\}$ for every $e \in E(H)$, but if $\sum_{v \in e} f(v) = 1$ for all $15$ edges $e$ in $H$, then $f$ would have to be the labeling in \cref{fig:label}, and hence not in $\Gamma_H(\S)$. 

Combining \cref{eq:143} with $\prod_{v \in V(H)} c_{f(v)} \gtrsim \delta$ from earlier, we deduce that $c_t \gtrsim \delta^{3t/14}$ for at least one of $t \in \{1/3, 1/2, 1\}$. Thus $\Vol(\S, c) = 3(c_1 + c_{1/2}^2 + c_{1/3}^3) \gtrsim \delta^{3/14}$.
\end{proof}

\section{Solution to the variational problem for cliques}
\label{sec:lowerbound}
To show \cref{thm:varsolutions}(a), we must give a construction to upper bound $\phi(H,n,p,\d)$ and prove a lower bound. The upper bound follows directly from the calculation in \cref{sec:future}, specifically \cref{ex:cliques}. In this section we show the lower bound.

\subsection{Preliminaries}

We recall some tools from \cite{LZ17}. The following inequality can be viewed as a generalized version of H\"older's inequality~\cite{Fin92}. It is also related to the Brascamp--Lieb inequalities~\cite{BL76}.

\begin{theorem}[Generalized H\"older's inequality]
\label{thm:holder}
For each $i \in [n]$, let $\Omega_i$ be a probability space with measure $\mu_i$. Let $\mu = \prod_{i=1}^n \mu_i$. Let $A_1, A_2, \dots, A_n$ be nonempty subsets of $[n] = \{1, 2, \dots, n\}$ and for $A \in [n]$ let $\mu_A = \prod_{i \in A} \mu_i$ and $\Omega_A = \prod_{i \in A} \Omega_i.$ Let $f_i \in L^{p_i}(\Omega_{A_i}, \mu_{A_i})$ for each $1 \le i \le m$. Assume further that $\sum_{i : j \in A_i} p_i^{-1} \le 1$ for all $1 \le j \le n.$ Then we have that
\begin{equation*}
\int \left( \prod_{i=1}^m f_i \right) \,\mathrm d\mu \le \prod_{i = 1}^m \left( \int |f_i|^{p_i} \,\mathrm d \mu_{A_i} \right)^\frac{1}{p_i}.
\end{equation*}
\end{theorem}
Our most common application involves the case where every element of $[n]$ is contained in at most $\Delta$ sets $A_j$. Then one can take
$p_i = \Delta$ for all $i \in [m]$, giving the inequality $\int f_1 \dots f_m \mathrm d\mu \le \prod_{i=1}^m \left(\int |f_i|^\D \mathrm d\mu_{A_i}\right)^\frac1\D.$ As an example use, we have for any function $f$ that \[ \int f(x, y)f(x, z)f(y, z)\,\dx\dy\dz  \le \left(\int f(x, y)^2 \,\dx\dy\right)^{3/2}. \]
More generally, we will apply \cref{thm:holder} in the following form.

\begin{corollary}[Generalized H\"older for bounded degree hypergraphs]
\label{cor:holder-bounded}
Let $H$ be a hypergraph with maximum degree at most $\Delta$. Let $\Omega$ be some probability space and $B_v \subset \Omega$ a measurable subset for each $v \in V(H)$. Let $U \colon \Omega^r \to \RR_{\ge 0}$ be a symmetric function. Then
\[
\int_{\prod_{v \in V(H)} B_v} \prod_{S \in E(H)} U(x_S) \, \dx_{V(H)}
 \le \prod_{S \in E(H)} \paren{ \int_{\prod_{v \in S} B_v} U(x_S)^{\Delta} \, \dx_S}^{1/\Delta}.
\]
\end{corollary}
As a further special case, we have the following inequality when $H$ is a clique. Here, for a set $S$, we write $\binom{S}{r}$ for the collection of $r$-element subsets of $S$.

\begin{corollary}[Generalized H\"older for cliques]
\label{cor:holder-clique}
Let $k \ge r$ be positive integers, Let $B$ be a measurable subset of some probability space $\Omega$. Let $U \colon \Omega^r \to \RR_{\ge 0}$ be a symmetric function. Then
\[
\int_{B^k} \prod_{S \in \binom{[k]}{r} } U(x_S) \, \dx_{[k]} 
\le \paren{ \int_{B^r} U(x_{[r]})^{\binom{k-1}{r-1}} \, \dx_{[r]} }^{k/r}
\]
\end{corollary}

\noindent Let $H = K_k^{(r)}$, let $G$ be an $r$-graph, and let $W$ be the weights corresponding to $G$ (with all weights in $[0,1]$) satisfying $t(H, W) \ge (1+\delta)p^{|E(H)|}.$
Define $U = W - p$. As $I_p$ is convex, decreasing on $0$ to $p$, and increasing on $p$ to $1$, we can assume that $0 \le U \le 1-p$, as we are trying to lower bound $I_p(W)$ over all $W$ with $t(H,W) \ge (1+\delta)p^{|E(H)|}.$ We can expand 
\begin{equation}
\label{eq:expand}
t(H, W) = \sum_{H' \subseteq H} p^{|E(H)| - |E(H')|} t(H', U)
\end{equation}
where the sum is over all subgraphs $H'$ of $H$. \\

\noindent For $b \in (0, 1]$, define the subset $B_b$ of vertices $x \in V(G)$ by
\[ B_b = B_b(G) := \{ x : d^G_2(x) > b \} \]
where \[ d_2(x) = d^G_2(x) := \int U(x, x_1, \dots, x_{r-1})^2 \dx_{[r-1]}. \]
From now on, we will drop the dependence on $G$ in $B_b(G)$ and $d^G_2(x)$ whenever $G$ and the corresponding functions $U$ and $W$ are clear from context.

\begin{remark}
In earlier work (e.g. \cite{LZ17,BGLZ17}), $d(x)$ was defined as $d(x) = \int U(x, y) \dy$ (representing the degree of $x$). While our proof actually relies on our different choice of $d_2$, it still seems natural because in our equality cases $U$ only takes on values $o(1)$ and $1 - o(1)$, so the two choices actually should agree.
\end{remark}
The next lemma is analogous to Lemma 4.2 in \cite{BGLZ17}.
\begin{lemma}
\label{lemma:graphonlemma}

Fix $H = K_k^{(r)}$ with $k > r \ge 3.$ For a graph $G$, let $U : V(G)^r \to \R$ be a function satisfying
\begin{equation} \label{eq:ipbound} \int I_p(p + U(x_{[r]})) \dx_{[r]} \lesssim p^{\binom{k-1}{r-1}}\log(1/p). \end{equation}
Then
\begin{equation}
\label{eq:l1bound1}
\int U(x_{[r]}) \dx_{[r]} \lesssim p^\frac{\binom{k-1}{r-1}+1}{2} \sqrt{\log(1/p)}
\end{equation}
and
\begin{equation}
\label{eq:46}
\int d_2(x) \dx = \int U(x_{[r]})^2 \dx_{[r]} \lesssim p^{\binom{k-1}{r-1}}.
\end{equation}
Furthermore, if $B_b = \{ x : d_2(x) > b \}$, then
\begin{equation}
\label{eq:47}
\lambda(B_b) \lesssim \frac{p^{\binom{k-1}{r-1}}}{b}
\end{equation}
where we write $\lambda(S) = \int_S 1 \dx_{[r]}$ for the measure of $S$.
\end{lemma}

Before giving the proof, we state some properties of the function $I_p(x)$ which were shown in \cite{LZ17}.
\begin{lemma}[{\cite[Lemma 3.3]{LZ17}}]
\label{lemma:approxip}
If $0 \le x \ll p$ then $I_p(p+x) \sim \frac{x^2}{2p}.$ If $p \ll x \le 1-p$ then $I_p(p+x) \sim x \log(x/p).$
\end{lemma}
\begin{lemma}[{\cite[Lemma 3.4]{LZ17}}]
\label{lemma:Iplower}
There exists $p_0 > 0$ so that for all $0 < p \le p_0$ and $0 \le x \le b \le 1 - p -1/\log(1/p)$,
\[ I_p(p + x) \ge (x/b)^2 I_p(p + b). \]
\end{lemma}
\begin{lemma}[{\cite[Corollary 3.5]{LZ17}}]
\label{lemma:x2}
There is a constant $p_0 > 0$ such that for all $0 < p \le p_0$ we have that
\begin{equation*}
I_p(p+x) \ge x^2 I_p(1-1/\log(1/p)) \sim x^2 \log(1/p).
\end{equation*}
\end{lemma}

\begin{proof}[Proof of \cref{lemma:graphonlemma}]
By \cref{lemma:approxip} we have 
\begin{equation*}
I_p\left(p + cp^\frac{\binom{k-1}{r-1}+1}{2} \sqrt{\log(1/p)} \right) \sim \frac{1}{2}c^2p^{\binom{k-1}{r-1}} \log(1/p)
\end{equation*}
Now, by the convexity of $I_p(\cdot)$ and \cref{eq:ipbound}, we find that
\begin{equation*}
I_p\left(p + \int U(x_{[r]}) \dx_{[r]} \right) \le \int I_p(p + U(x_{[r]})) \dx_{[r]} \lesssim p^{\binom{k-1}{r-1}} \log(1/p)
\end{equation*}
Now applying the monotonicity of $I_p(x)$ for $x \ge p$ gives us that
\begin{equation*}
\int U(x_{[r]}) \dx_{[r]} \lesssim p^\frac{\binom{k-1}{r-1}+1}{2} \sqrt{\log(1/p)}
\end{equation*}
as desired.
By \cref{lemma:x2} we have that
\begin{equation*}
\int U(x_{[r]})^2 \dx_{[r]} \lesssim \frac{1}{\log(1/p)} \int I_p(p + U(x_{[r]})) \dx_{[r]} \lesssim p^{\binom{k-1}{r-1}}
\end{equation*}
as desired.
Finally it is clear from the definition of $B_b$ that
\[ \lambda(B_b) b \le \int U(x_{[r]})^2 \dx_{[r]} \lesssim p^{\binom{k-1}{r-1}},\]
so
\[ \lambda(B_b) \lesssim \frac{p^{\binom{k-1}{r-1}}}{b}. \qedhere
 \] 
\end{proof}

\subsection{Lower bound for the variational problem}
In this section, we prove the lower bound to \cref{thm:varsolutions}(a). We assume \cref{eq:ipbound}, i.e., $\int I_p(p+U(x_{[r]})) \dx_{[r]} \lesssim p^{\binom{k-1}{r-1}}\log(1/p)$, or else we are done.
Our first step towards showing \cref{thm:rate}(a) is eliminating the negligible terms in \cref{eq:expand}.
\begin{lemma}[Negligible terms]
\label{lemma:neg}
Fix $H = K_k^{(r)}$ for $k > r \ge 3$. Let $U$ be a function satisfying \cref{eq:ipbound}. Let $S_k^{(r)}$ (a star) denote the $k$-vertex $r$-graph that consists of all the edges that contain some fixed vertex. If $H'$ is a subgraph of $H$ such that $H' \ncong H$ and $H' \ncong S_k^{(r)}$, then $t(H', U) \ll p^{|e(H')|}.$
\end{lemma}

We will prove \cref{lemma:neg} after the following hypergraph theoretic lemma.

\begin{lemma}
\label{lemma:graphtheory2}
Let $G$ be a nonempty $r$-graph on $k$ vertices other than $K_k^{(r)}$ and $S_k^{(r)}$ for $k > r \ge 2$. Then there exists a subgraph $G'$ of $G$ such that
\[ |E(G')| > \frac{\binom{k-1}{r-1} - 1}{\binom{k-1}{r-1}}|E(G)| \text{ and } \Delta(G') \le \binom{k-1}{r-1} - 1. \]
\end{lemma}
\begin{proof}
Define $T = \left\{ v \in G : \deg(v) = \binom{k-1}{r-1} \right\}$, and let $t = |T|.$ If $t = 0$ then setting $G' = G$ suffices.
If $t = 1$ and $G \neq S_k^{(r)}$ (as given by the condition) we can take $G'$ to be $G$ with a single edge containing the vertex in $T$ removed. Clearly, $\D(G') \le \binom{k-1}{r-1} - 1$ and
\[ |E(G')| = |E(G)| - 1 > \frac{\binom{k-1}{r-1} - 1}{\binom{k-1}{r-1}}|E(G)| \] as $|E(G)| > \binom{k-1}{r-1}$ because $G \neq S_k^{(r)}$.

Otherwise, the condition $G \neq K_k^{(r)}$ shows that $2 \le t \le k-r.$
Note that the $r$-graph induced by $T$ on $G$ is complete. Therefore, one can remove $\left\lceil t/r \right\rceil$ edges from $G$, such that all edges but at most one touch $r$ vertices in $T$, to get a subgraph $G'$ such that $\Delta(G') \le \binom{k-1}{r-1} - 1.$
Therefore, it suffices to verify that
\begin{equation*}
|E(G)| - \left\lceil \frac{t}{r} \right\rceil > \frac{\binom{k-1}{r-1} - 1}{\binom{k-1}{r-1}} \cdot |E(G)|.
\end{equation*}
Using $\left\lceil t/r \right\rceil \le \frac{t+r-1}{r}$ and $|E(G)| \ge \binom{k}{r} - \binom{k-t}{r}$ (which follows from the definition of $T$), it suffices to verify that
\begin{equation*}
\binom{k}{r} - \binom{k-t}{r} > \frac{t+r-1}{r}  \binom{k-1}{r-1}.
\end{equation*}
This reduces to \[ \frac{k-r+1-t}{r} \binom{k-1}{r-1} > \frac{k-r+1-t}{r} \binom{k-t}{r-1}, \] which is true for $2 \le t \le k-r.$
\end{proof}
We use an easy consequence of \cref{lemma:graphtheory2}.
\begin{corollary}
\label{lemma:graphtheory}
Let $G$ be a nonempty $r$-uniform hypergraph on $k$ vertices other than $K_k^{(r)}$ and $S_k^{(r)}$ for $k > r \ge 3$. Then there exists a subgraph $G'$ of $G$ such that
\[ \frac{|E(G')|}{\max(2, \D(G'))} > \frac{|E(G)|}{\binom{k-1}{r-1}}. \]
\end{corollary}
\begin{proof}
If $\D(G) < \binom{k-1}{r-1}$ we are done by taking $G' \cong G.$ If $\D(G) = \binom{k-1}{r-1}$, then take $G'$ to be $r$-graph in the conclusion of \cref{lemma:graphtheory2}. Then because $\binom{k-1}{r-1}-1 \ge 2$ by $k > r \ge 3$, we have \[ \frac{|E(G')|}{\max(2,\D(G'))} \ge \frac{|E(G')|}{\binom{k-1}{r-1} - 1} > \frac{|E(G)|}{\binom{k-1}{r-1}} \] as desired.
\end{proof}
\begin{proof}[Proof of \cref{lemma:neg}]
For $H' \ncong K_k^{(r)}$ and $H' \ncong S_k^{(r)}$, let $H''$ be a subgraph of $H'$ satisfying the properties of \cref{lemma:graphtheory}. Then \[ \frac{|E(H'')|}{\max(2, |\Delta(H'')|)} > \frac{|E(H')|}{\binom{k-1}{r-1}}. \] By \cref{cor:holder-bounded} we have that
\begin{align*}
t(H', U) \le t(H'', U) &\le \left( \int U(x_{[r]})^{\max(2, \Delta(H''))} \dx_{[r]} \right)^\frac{|E(H'')|}{\max(2, \Delta(H''))} \\
					& \ll \left( \int U(x_{[r]})^{\max(2, \Delta(H''))} \dx_{[r]} \right)^\frac{|E(H')|}{\binom{k-1}{r-1}} \lesssim p^{|E(H')|},
\end{align*}
where we used that 
\begin{equation*} 
\int U(x_{[r]})^{\max(2, \Delta(H''))} \dx_{[r]} \le \int U(x_{[r]})^2 \dx_{[r]} \lesssim p^{\binom{k-1}{r-1}}
\end{equation*} by \cref{eq:46}.
\end{proof}

We proceed towards the proof of \cref{thm:varsolutions}(a).

Fix $H = K_k^{(r)}.$ Let $W: V(G)^r \to [p, 1]$ be a symmetric function satisfying $t(H, W) \ge (1+\delta)p^{\binom{k}{r}}.$ Now, we set $U = W-p$ and assume that $U$ satisfies \cref{eq:ipbound} or else we are already done. By \cref{lemma:neg}, we know that
\begin{equation*}
(1+\delta)p^{\binom{k}{r}} \le t(H, W) = \sum_{H' \subseteq H} p^{|E(H)| - |E(H')|} t(H', U) = p^{\binom{k}{r}} + t(K_k^{(r)}, U) + k p^{\binom{k-1}{r}} t(S_k^{(r)}, U) + o(p^{\binom{k}{r}}).
\end{equation*}
Therefore, we have that
\begin{equation}
\label{eq:afterneg}
t(K_k^{(r)}, U) + k p^{\binom{k-1}{r}} t(S_k^{(r)}, U) \ge (\delta - o(1))p^{\binom{k}{r}}.
\end{equation}
Recall that for a constant $b$, we define $B_b = \{ x : d_2(x) > b \}.$ For the purposes of the proof, we
let $b$ satisfy $p^\eps \ll b \ll 1$ for a sufficiently small constant $\eps$ depending on $k$ and $r$.
Before continuing, we define two quantities:
\begin{equation*}
\theta_b = \int_{\overline{B}_b^r} U(x_{[r]})^2 \dx_{[r]} \quad \text{and} \quad  \eta_b = \int_{B_b \times \overline{B}_b^{r-1}} U(x_{[r]})^2 \dx_{[r]}.
\end{equation*}

Now, we analyze the $t(K_k^{(r)}, U)$ term first. We show the following bound.
\begin{lemma}
\label{lemma:completebound}
Let $U: V(G)^r \to [0, 1]$ be a symmetric function satisfying \cref{eq:ipbound}. Then we have for $k > r \ge 3$ that
\begin{equation*}
t(K_k^{(r)}, U) \le \theta_b^{k/r} + o(p^{\binom{k}{r}}).
\end{equation*}
\end{lemma}
\begin{proof}
Let $\binom{[k]}{r}$ denote the set of all subsets of $[k]$ of size $r$. Note that
\begin{align}
\int_{B_b \times V(G)^{k-1}} \prod_{S \in \binom{[k]}{r}} U(x_S) \dx_{[k]} &\le \lambda(B_b) t(K_{k-1}^{(r)}, U) \le \lambda(B_b) \left( \int U(x_{[r]})^{\binom{k-2}{r-1}} \dx_{[r]} \right)^\frac{\binom{k-1}{r}}{\binom{k-2}{r-1}} \nonumber \\
& \le \lambda(B_b) \left( \int U(x_{[r]})^{\binom{k-2}{r-1}} \dx_{[r]} \right)^\frac{k-1}{r} \label{eq:422}
\end{align}
by \cref{cor:holder-clique}.

If $k = r+1$, we can compute from the above that
\begin{align*}
\lambda(B_b) \left( \int U(x_{[r]})^{\binom{k-2}{r-1}} \dx_{[r]} \right)^\frac{k-1}{r} \le \lambda(B_b)  \int U(x_{[r]}) \dx_{[r]} \lesssim \frac{p^r}{b} \cdot p^\frac{r+1}{2} \sqrt{\log(1/p)} \ll p^{r+1}
\end{align*}
by \cref{eq:l1bound1}, as desired. Otherwise, we have that the expression in \cref{eq:422} is
\begin{align*}
\lambda(B_b) \left( \int U(x_{[r]})^{\binom{k-2}{r-1}} \dx_{[r]} \right)^\frac{k-1}{r} &\le \lambda(B_b) \left( \int U(x_{[r]})^{\binom{k-2}{r-1}} \dx_{[r]} \right)^\frac{k-1}{r} \\ 
&\le \lambda(B_b) \left(\int U(x_{[r]})^2 \dx_{[r]}\right)^\frac{k-1}{r} \lesssim \frac{p^{\binom{k-1}{r-1}}}{b} \cdot p^{\frac{k-1}{r} \binom{k-1}{r-1}} \ll p^{\binom{k}{r}}
\end{align*}
by \cref{eq:47}, as desired.

Finally, we have that
\begin{equation*}
t(K_k^{(r)}, U) =  \int_{\overline{B}_b^k} \prod_{S \in \binom{[k]}{r}} U(x_S) \dx_{[k]} + o(p^{\binom{k}{r}}) \le \left(\int_{\overline{B}_b^r}U(x_{[r]})^{\binom{k-1}{r-1}} dx_{[r]} \right)^{k/r} + o(p^{\binom{k}{r}}) \le \theta_b^{k/r} + o(p^{\binom{k}{r}})
\end{equation*}
by \cref{cor:holder-clique}.
\end{proof}

Now, we analyze the $t(S_k^{(r)}, U)$ term. We show the following bound.
\begin{lemma}
\label{lemma:starbound}
Let $U: V(G)^r \to [0, 1]$ be a symmetric function satisfying \cref{eq:ipbound}. Then we have for $k > r \ge 3$ that
\begin{equation*}
t(S_k^{(r)}, U) \le \eta_b + o(p^{\binom{k-1}{r-1}}).
\end{equation*}
\end{lemma}

\begin{proof}
Let $T = \{ S \subseteq [k] : 1 \in S \text{ and } |S| = r \}.$
By \cref{cor:holder-clique} and \cref{lemma:graphonlemma} we have that
\begin{align}
\label{eq:r2}
\int_{\overline{B}_b \times V(G)^{k-1}} \prod_{S \in T} U(x_S) \dx_{[k]} &\le \int_{\overline{B}_b} \left(\int U(x_{[r]})^{\binom{k-2}{r-2}} \dx_2 \dx_3 \dots \dx_r \right)^\frac{k-1}{r-1} \dx_1 \\
\label{eq:r3}
&\le \int_{\overline{B}_b} d_2(x)^\frac{k-1}{r-1} \dx \le b^\frac{k-r}{r-1} \cdot \int d_2(x) \dx\\
&\lesssim p^{\binom{k-1}{r-1}} b^\frac{k-r}{r-1} \ll p^{\binom{k-1}{r-1}}. \nonumber
\end{align}
We remark that going from \cref{eq:r2} to \cref{eq:r3} required $r \ge 3$ to get that $\binom{k-2}{r-2} \ge 2$ for $k > r \ge 3.$
By \cref{lemma:graphonlemma} we also have that
\begin{equation*}
\int_{B_b \times B_b \times V(G)^{k-2}} \prod_{S \in T} U(x_S) \dx_{[k]} \le \lambda(B_b)^2 \lesssim \frac{p^{2\binom{k-1}{r-1}}}{b^2} \ll p^{\binom{k-1}{r-1}}.
\end{equation*}
Thus by \cref{cor:holder-clique} again,
\begin{align*}
t(S_k^{(r)}, U) &\le \int_{B_b \times \overline{B}_b^{k-1}} \prod_{S \in T} U(x_S) \dx_{[k]} + o(p^{\binom{k-1}{r-1}}) \\
&\le  \int_{B_b} \left(\int_{\overline{B}_b^{r-1}} U(x_{[r]})^{\binom{k-2}{r-2}} \dx_2 \dx_3 \cdots \dx_r \right)^\frac{k-1}{r-1} \dx_1 + o(p^{\binom{k-1}{r-1}}) \\
&\le \int_{B_b} \left(\int_{\overline{B}_b^{r-1}} U(x_{[r]})^2 \dx_2 \dx_3 \cdots \dx_r \right)^\frac{k-1}{r-1} \dx_1 + o(p^{\binom{k-1}{r-1}}) \\
&\le \int_{B_b \times \overline{B}_b^{r-1}} U(x_{[r]})^2 \dx_{[r]} + o(p^{\binom{k-1}{r-1}}) = \eta_b + o(p^{\binom{k-1}{r-1}})
\end{align*}
as desired.
\end{proof}

\begin{proof}[Proof of \cref{thm:varsolutions}(a)]
Note that by \cref{lemma:x2} we have that
\begin{equation}
\label{eq:thetaeta}
\theta_b + r\eta_b \le \int U(x_{[r]})^2 \dx_{[r]} \le (1+o(1)) \cdot \frac{1}{\log(1/p)} \int I_p(p + U(x_{[r]})) \dx_{[r]}.
\end{equation}

Combining \cref{lemma:completebound}, \cref{lemma:starbound}, \cref{eq:afterneg}, and \cref{eq:thetaeta} yields that
\begin{align}
&\int I_p(p + U(x_{[r]})) \dx_{[r]} \nonumber \\ &\ge (1-o(1)) \log(1/p) \inf\left\{ x+ry : x^{k/r} + kp^{\binom{k-1}{r}} y \ge (\delta+o(1))p^{\binom{k}{r}} \text{ and } x, y \ge 0 \right\} \label{eq:fin} \\
&= (1-o(1))\min\left( \delta^{r/k}, \frac{r}{k} \delta \right) p^{\binom{k-1}{r-1}} \log(1/p) \nonumber
\end{align}
by convexity of the functions $x^{k/r}$ and $y$ (e.g., see \cite[Lemma~5.4]{BGLZ17}).
\end{proof}

\section{Order of the rate function}
\label{sec:orderrate}
We now show \cref{thm:asympbound}, showing that the rate function conjectured in \cref{conj:main} has the correct order of magnitude.
\begin{proof}[Proof of \cref{thm:asympbound}]
Assume $\Delta \ge 2$, as $\Delta = 1$ is trivial. The upper bound follows from \cref{lemma:conjupper}. For the lower bound, we use the same notation as in \cref{sec:lowerbound}. Let $W$ be the weighted adjacency array of an $r$-graph $G$ such that $t(H, W) \ge (1+\delta)p^{|E(H)|}.$ Letting $U = W-p$, we can expand \[ t(H, W) = \sum_{H' \subseteq H} p^{|E(H)-|E(H')|} t(H', U). \] Recall that \[ t(H', U) \le \left( \int U(x_{[r]})^\D \dx_{[r]} \right)^\frac{|E(H')|}{\D} \le \left( \int U(x_{[r]})^2 \dx_{[r]} \right)^\frac{|E(H')|}{\D} \] by \cref{cor:holder-bounded}. Therefore, we get that \[ (1+\d)p^{|E(H)|} \le \sum_{H' \subseteq H} p^{|E(H)-|E(H')|} t(H', U) \ls \sum_{H' \subseteq H} p^{|E(H)-|E(H')|} \left( \int U(x_{[r]})^2 \dx_{[r]} \right)^\frac{|E(H')|}{\D}. \] This implies that $\int U(x_{[r]})^2 \dx_{[r]} \gtrsim p^\Delta.$ Now by \cref{lemma:x2}, \[ \int I_p(p+U(x_{[r]})) \dx_{[r]} \ge (1-o(1)) \log(1/p) \int U(x_{[r]})^2 \dx_{[r]} \gs p^\D \log(1/p). \qedhere \]
\end{proof}

\section{Solution to the Variational Problem for a Special Graph}
\label{sec:quad}

In this section we prove \cref{thm:varsolutions}(b), solving the variational problem for the 3-graph in \cref{fig:graph1}, reproduced below:
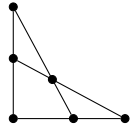
\[
\begin{tikzpicture}[scale=.4, every node/.style={draw, circle, black, fill, inner sep = 0pt, minimum width = 3pt},font=\footnotesize]
	\node (1) at (0, 0) {};
	\node (2) at (2, 0) {};
	\node (3) at (0, 2) {};
	\node (4) at (1.3, 1.3) {};
	\node (5) at (intersection of 1--2 and 3--4) {};
	\node (6) at (intersection of 1--3 and 2--4) {};
	\draw (2)--(6)--(1)--(5)--(3);
\end{tikzpicture}
\]

We have resolved the upper bound in \cref{sec:future}, specifically \cref{ex:quad}. For the lower bound, we develop slightly more general techniques than those used in \cref{sec:lowerbound}. We use the same notation as in \cref{sec:lowerbound}.

\subsection{Analogues and extensions of \cref{lemma:graphonlemma}}
For this section, let $H$ be any $r$-graph with maximum degree $\Delta.$ Recall that $\G_n$ is the family of all edge-weighted $r$-graphs on $n$ vertices with weights in $[0, 1].$ Throughout, we let $W: [n]^r \to [0, 1]$ denote the adjacency array of a graph in $\G_n$. We show that when solving the discrete variational problem, it suffices to consider only a subset of the weighted $r$-graphs in $\G_n$, where there are no $\xr \in [n]^r$ satisfying $p < W(\xr) \le (1+o(1))p.$ In other words, there are no $\xr \in [n]^r$ such that $W(\xr)$ has value very close to $p$ but not equal to $p$. This is made precise by the following lemma.
\begin{lemma}
\label{lemma:redux}
Let $H$ be an $r$-graph, and let $\delta$ be fixed. Define, for every $\ka > 0$,
\begin{align*} \phi_{\ka}(H, n, p, \delta) = \inf \big\{I_p(W) : W& \in \G_n \text{ with } t(H, W) \ge (1+\d)p^{|E(H)|} \\ &\text{ and } W(\xr) \in \{p\} \cup [(1+\ka)p, 1] \forall \xr \in [n]^r \big\}. \end{align*} 
Then
\[
\phi(H, n, p, \d) \ge \phi_{\ka}(H, n, p, (1+\ka)^{-|E(H)|}(1+\d)-1).
\]
\end{lemma}

\begin{proof}
Given a symmetric $W \colon [n]^r \to [0,1]$ with $t(H, W) \ge (1+\delta)p^{|E(H)|}$, set $W': [n]^r \to [0, 1]$ by
\[
W'(\xr) = \begin{cases}
p, & \text{if } W(\xr) \le (1+\ka)p, \\
W(\xr), & \text{otherwise.}
\end{cases}
\]
We have $W \le (1+\kappa) W'$ and $I_p(W(\cdot)) \ge I_p(W'(\cdot))$ both holding pointwise.
So $t(H,W) \le (1+\ka)^{|E(H)|}t(H,W')$, and thus $W'$ satisfies the constraints in the definition of $\phi_{\ka}(H, n, p, (1+\ka)^{-|E(H)|}(1+\d)-1)$. The claimed inequality then follows.
\end{proof}

We use \cref{lemma:redux} in the setting where $\ka := \ka(p) = o(1)$ as $p \to 0.$ In this way, we have that for some $\delta' = \delta - o(1)$ that
\[ \phi(H, n, p, \d) \ge \phi_{\ka}(H, n, p, (1+\ka)^{-|E(H)|}(1+\d)-1) = \phi_{\ka}(H, n, p, \delta'). \] We now focus on lower bounding $\phi_{\ka}(H,n,p,\d)$.

Let $G$ be some $r$-graph on $n$ vertices. We first state an extension of \cref{lemma:graphonlemma} that we will use.
\begin{lemma}
\label{lemma:graphonlemmaextend}
Let $U: V(G)^r \to [0, 1]$ be a symmetric function satisfying
\begin{equation} \label{eq:ipbound2} \int I_p(p + U(x_{[r]})) \dx_{[r]} \lesssim p^\Delta \log(1/p). \end{equation}
For $x \in V(G)$, define $d(x) = \int U(x, x_1, \dots, x_{r-1}) \dx_{[r-1]}.$ Then we have that
\begin{equation}
\label{eq:l2bound}
\int U(x_{[r]})^2 \dx_{[r]} \lesssim p^\D.
\end{equation}
Let $\eps > 0$ be a fixed parameter (not depending on $p$), and let $b, b'$ be parameters such that $b' < b$. If we define $B = \{x \in V(G): b' \le d(x) \le b \}$ then we have that if $p^{1-\eps} \ll b$ then
\begin{equation}
\label{eq:d2bound1}
\int_B d(x)^2 \dx\ls_\eps p^\D b
\end{equation}
and if $p^{1-\eps} \ll b'$ then
\begin{equation}
\label{eq:lam1}
\lambda(B) \ls_\eps \frac{p^\D}{b'}
\end{equation}
where $\lambda(B) = \int_B 1 \dx.$ Here $\ls_\eps$ denotes that the constant in the $\ls$ depends on $\eps.$
\end{lemma}
Now, for the remainder of the section, we write $\ka := \ka(p) \to 0$, where the dependence on $p$ is implicit.
Additionally, we only consider the case \begin{align*} &W(\xr) \in \{p\} \cup [(1+\ka )p, 1] \forall \xr\in V(G)^r \\ \text{ and } &U(\xr) \in \{0\} \cup [\ka p, 1-p] \forall \xr\in V(G)^r, \end{align*} where $U = W-p$ and $\ka $ is an arbitrary function satisfying $\ka  = o(1)$ as $p \to 0.$ In this setting, we can state a different extension of \cref{lemma:graphonlemma} that we will also use.
\begin{lemma}
\label{lemma:graphonlemmaextend2}
Let $\ka $ be an arbitrary function satisfying $\ka  = o(1)$ as $p \to 0.$ Let $U: V(G)^r \to [0, 1]$ be a symmetric function satisfying 
\begin{equation} \label{eq:ipbound22} U(\xr) \in \{0\} \cup [\ka p, 1-p] \forall \xr\in V(G)^r \end{equation}
and
\begin{equation*} \int I_p(p + U(x_{[r]})) \dx_{[r]} \lesssim p^\Delta \log(1/p) \end{equation*}
For $x \in V(G)$, define $d(x) = \int U(x, x_1, \dots, x_{r-1}) \dx_{[r-1]}.$ Then we have that
\begin{equation}
\label{eq:l1bound}
\int d(x) \dx = \int U(\xr) d\xr \lesssim \frac{p^\D \log(1/p)}{\ka }.
\end{equation}
For arbitrary parameters $b, b'$ (not necessarily $b, b' \gg p$) if we define $B = \{x \in V(G) : b' \le d(x) \le b \}$ then we have that
\begin{equation}
\label{eq:d2bound}
\int_B d(x)^2 \dx \ls \frac{p^\D \log(1/p)b}{\ka }
\end{equation}
and
\begin{equation}
\label{eq:lam2}
\lambda(B) \ls \frac{p^\D \log(1/p)}{b' \ka },
\end{equation}
where $\lambda(B) = \int_B 1 \dx.$
\end{lemma}
Before proceeding to the proofs, we point out some differences between \cref{lemma:graphonlemma} versus \cref{lemma:graphonlemmaextend,lemma:graphonlemmaextend2}. The main differences comes from a stronger bound on $\int U(\xr) d\xr$ (compare \cref{eq:l1bound} and \cref{eq:l1bound1}). This stronger bound allowed us to bound $\lambda(B)$ and $\int_B d(x)^2 \dx$ even in the case where $B = \{x \in V(G) : b' \le d(x) \le b \}$ and $b, b' \ll p$. Essentially, \cref{lemma:redux} allowed us to restrict our attention to only a subset of weighted $r$-graphs in $\G_n$, and we are able to achieve better bounds on say $\int U(\xr) d\xr$ for this subset.

We now proceed to proofs of \cref{lemma:graphonlemmaextend} and \cref{lemma:graphonlemmaextend2}.
\begin{proof}[Proof of \cref{lemma:graphonlemmaextend}]
By \cref{lemma:x2} we have that
\[
\int U(x_{[r]})^2 \dx_{[r]} \lesssim \frac{1}{\log(1/p)} \int I_p(p + U(x_{[r]})) \dx_{[r]} \lesssim p^\D,
\]
which proves \cref{eq:l2bound}.
It is clear from the definition of $B$ (in the case $b', b \gg p^{1-\eps}$) and convexity that
\[ \lambda(B) I_p(p+b') \le \int_B I_p(p+d(x)) \dx \le \int I_p(p+U(x_{[r]})) d\xr \ls p^\D \log(1/p), \]
so
\[ \lambda(B) \ls \frac{p^\D \log(1/p)}{I_p(p+b')} \ls_\eps \frac{p^\D}{b'} \] by \cref{lemma:approxip}. This shows \cref{eq:lam1}. We used $b' \gg p^{1-\eps}$ to obtain $I_p(p+b') \gs b' \log(b'/p) \gs_\eps b' \log(1/p).$ Now, by \cref{lemma:Iplower}, we have that
\[ I_p(p+b) \int_B (d(x)/b)^2 \dx \le \int_B I_p(p+d(x)) \dx \le \int I_p(p+U(\xr)) d\xr \ls p^\D \log(1/p), \]
so by \cref{lemma:approxip},
\[ \int_B d(x)^2 \dx \ls \frac{b^2 p^\D \log(1/p)}{I_p(p+b)} \ls_\eps p^\D b, \] which shows \cref{eq:d2bound1}. We used $b \gg p^{1-\eps}$ to obtain $I_p(p+b) \gs b \log(b/p) \gs_\eps b \log(1/p).$
\end{proof}

\begin{proof}[Proof of \cref{lemma:graphonlemmaextend2}]
We first argue that for all $t \in \{0\} \cup [\ka p, 1-p]$ we have that $I_p(p+t) \gs \ka t.$ Indeed, this holds for $t = 0.$ For $t = \ka p$, we have by \cref{lemma:approxip} that \[ I_p(p + t) = I_p(p + \ka p) \gs \frac{(\ka p)^2}{p} = \ka t. \] Now, by convexity of $I_p(p+x)$ we have that $I_p(p+t) \gs \ka t$ for all $t \ge \ka p$ as desired. Using this, we get that
\[ \int d(x) \dx = \int U(\xr) d\xr \ls \int \frac{I_p(p + U(\xr))}{\ka } d\xr \ls \frac{p^\D \log(1/p)}{\ka } \] which shows \cref{eq:l1bound}.
Now, \[ \int_B d(x)^2 \dx \le b \int_B d(x) \dx \ls \frac{p^\D \log(1/p)b}{\ka } \] by the above. This shows \cref{eq:d2bound}.
Also, we have that \[ \lambda(B)b' \le \int_B d(x) \dx \ls \frac{p^\D \log(1/p)}{\kappa}, \]
so
\[ \lambda(B) \ls \frac{p^\D \log(1/p)}{b' \ka }, \] which shows \cref{eq:lam2}.
\end{proof}

\subsection{Proof of \cref{thm:varsolutions}(b)}
Let $H$ be the 3-graph in \cref{fig:graph1}, and $\D = 2$. Let $W:V(G)^3 \to [0, 1]$ be a symmetric function satisfying $t(H, W) \ge (1+\d)p^4$, and let $U = W-p.$ We assume that $U$ satisfies \cref{eq:ipbound2} or else we are already done.

To this end, let $b_1, b_2$ be parameters so that $p^{2\eps_0} \ll b_1, b_2 \ll p^{\eps_0}$ for some fixed sufficiently small $\eps_0 > 0$ (say $\eps_0 = \frac{1}{100}$). Define \[ B_1 = \{ x \in V(G) : d(x) \ge b_1 \}, B_2 = \{ x \in V(G) : pb_2 \le d(x) < b_1 \}, B_3 = \{ x \in V(G) : d(x) < pb_2 \}, \] where $d(x) = \int U(x, y, z) \dy\,\dz$ as defined above. Our use of $B_3$ here is novel, invoking simultaneous thresholds at very different scales (apart by a factor of nearly $p$), and it has not appeared in previous analyses in the graph setting~\cite{LZ17,BGLZ17}. The use of $B_3$ appears to be essential to our argument.

Define
\begin{equation} \label{eq:theta1} \theta_1 = \int_{B_1 \times B_3^2} U(x, y, z)^2 \dx\,\dy\,\dz \end{equation}
\begin{equation} \label{eq:theta2} \theta_2 = \int_{B_2^2 \times B_3} U(x, y, z)^2 \dx\,\dy\,\dz \end{equation}
\begin{equation} \label{eq:theta3} \theta_3 = \int_{B_3^3} U(x, y, z)^2 \dx\,\dy\,\dz \end{equation}
\begin{equation} \label{eq:eta} \eta = \int_{B_2 \times B_3^2} U(x, y, z)^2 \dx\,\dy\,\dz. \end{equation}

As in \cref{sec:lowerbound}, we can write $W = U + p$ and expand \[ t(H, W) = p^4 + 4p^3 t(E_1, U) + 6p^2 t(E_2, U) + 4p t(E_3, U) + t(H, U) \ge (1+\d)p^4 \] where $E_i$ is the subgraph of $H$ with exactly $i$ edges:
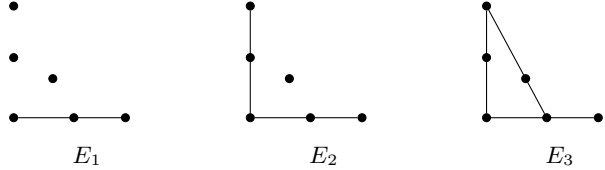
\[
\begin{tikzpicture}[baseline=(current bounding box.center),scale=.4, every node/.style={draw, circle, black, fill, inner sep = 0pt, minimum width = 3pt},font=\footnotesize]
	\node (1) at (0, 0) {};
	\node (2) at (2, 0) {};
	\node (3) at (0, 2) {};
	\node (4) at (1.3, 1.3) {};
	\node (5) at (intersection of 1--2 and 3--4) {};
	\node (6) at (intersection of 1--3 and 2--4) {};
	\draw (1)--(5);
	\node[draw=none,circle=none,fill=none,below left =3mm and 3mm of 5] {$E_1$};
\end{tikzpicture}
\qquad\qquad  
\begin{tikzpicture}[baseline=(current bounding box.center),scale=.4, every node/.style={draw, circle, black, fill, inner sep = 0pt, minimum width = 3pt},font=\footnotesize]
	\node (1) at (0, 0) {};
	\node (2) at (2, 0) {};
	\node (3) at (0, 2) {};
	\node (4) at (1.3, 1.3) {};
	\node (5) at (intersection of 1--2 and 3--4) {};
	\node (6) at (intersection of 1--3 and 2--4) {};
	\draw (5)--(1)--(6);
	\node[draw=none,circle=none,fill=none,below left =3mm and 3mm of 5] {$E_2$};
\end{tikzpicture}
\qquad\qquad  
\begin{tikzpicture}[baseline=(current bounding box.center),scale=.4, every node/.style={draw, circle, black, fill, inner sep = 0pt, minimum width = 3pt},font=\footnotesize]
	\node (1) at (0, 0) {};
	\node (2) at (2, 0) {};
	\node (3) at (0, 2) {};
	\node (4) at (1.3, 1.3) {};
	\node (5) at (intersection of 1--2 and 3--4) {};
	\node (6) at (intersection of 1--3 and 2--4) {};
	\draw (5)--(1)--(6)--(2);
	\node[draw=none,circle=none,fill=none,below left =3mm and 3mm of 5] {$E_3$};
\end{tikzpicture}
\]

We now analyze each piece separately. We extensively use \cref{lemma:graphonlemmaextend}.
When we use \cref{lemma:graphonlemmaextend}, we use the parameters $\eps = \frac{1}{100}$ (fixed small constant) and $\ka = 1/\log(1/p)$ (recall that in our notation $\ka$ depends implicitly on $p$).
The choice of $\ka $ is simply a natural explicit function that goes to $0$ slowly, and $1/\log(1/p)$ satisfies that property.

An additional tool we employ is the idea of \emph{adaptive thresholding}, which was introduced in \cite{BGLZ17}. When we are bounding the contribution of $t(H, U)$, we do not assume that the parameters $b_1, b_2$ are fixed. Instead, we allow them to depend on $U$. Therefore, our claim will instead say that there \emph{exists} a choice of $b_1$ and $b_2$ that allows us to get a sufficiently strong bound on $t(H, U).$ This is used in \cref{lemma:h} through the application of \cref{lemma:xb1} and \cref{lemma:xb2}.

For the analysis below, recall the following from \cref{lemma:graphonlemmaextend} and \cref{lemma:graphonlemmaextend2}. We have that
\begin{equation} \label{eq:lambound} \lambda(B_1) \ls \frac{p^2}{b_1} \text{ and } \lambda(B_2) \ls \frac{p \log(1/p)}{\ka b_2} \text{ and } \lambda(\oB_3) \ls \frac{p \log(1/p)}{\ka b_2}, \end{equation} where the final claim follows from $\lambda(\oB_3) = \lambda(B_1) + \lambda(B_2).$

\begin{lemma}[Analysis of $t(E_1, U)$]
\label{lemma:e1}
Let $U:V(G)^3 \to [0, 1]$ be a symmetric function satisfying \cref{eq:ipbound2}. We have that $t(E_1, U) = o(p).$
\end{lemma}
\begin{proof}
By \cref{eq:l1bound} we have that \[ t(E_1, U) = \int U(x, y, z) \dx\,\dy\,\dz \ls \frac{p^2 \log(1/p)}{\ka } \ll p. \qedhere \] 
\end{proof}
\begin{lemma}[Analysis of $t(E_2, U)$]
\label{lemma:e2}
Let $U:V(G)^3 \to [0, 1]$ be a symmetric function satisfying \cref{eq:ipbound2}. 
For parameters $b_1$ and $b_2$ satisfying $p^{2\eps_0}\ll b_1,b_2 \ll p^{\eps_0}$ we have that $t(E_2, U) \le \theta_1 + o(p^2).$
\end{lemma}
\begin{proof}
Recall that \begin{equation} \label{eq:e2} t(E_2, U) = \int U(x, y, z)U(x, y', z') \dx\,\dy\,\dz\,\dy'\,\dz' = \int d(x)^2 \dx. \end{equation}
We claim that all contribution to $t(E_2, U)$ where $x \not\in B_1$ is $o(p^2)$. Indeed, we have by \cref{eq:d2bound1} with the choice $b = b_1$ and $b' = 0$ that \[ \int_{\oB_1} d(x)^2 \dx \ls p^2b_1 \ll p^2 \] as $b_1 \ll p^{\eps_0}.$

Now we consider when $x \in B_1.$ We claim the contribution to $t(E_2,U)$ from the region where any of $y, z, y', z'$ are in $\oB_3$ has $o(p^2)$ contribution in \cref{eq:e2}.
Without loss of generality, say that $y$ lies in $\oB_3.$ These contributions are bounded by
\[ \int_{B_1 \times \oB_3 \times V(G)^3} U(x, y, z)U(x, y', z') \dx\,\dy\,\dz\,\dy'\,\dz' \le \lambda(B_1)\lambda(\oB_3) \ls \frac{p^2}{b_1} \cdot \frac{p \log(1/p)}{\ka b_2} \ll p^2, \] where we have used \cref{eq:lambound} and $b_1, b_2 \gg p^{2\eps_0}$.

Therefore
\begin{align*} t(E_2, U) &= \int_{B_1 \times B_3^4} U(x, y, z)U(x, y', z') \dx\,\dy\,\dz\,\dy'\,\dz' + o(p^2) 
\\ &\le \int_{B_1 \times B_3^2} U(x, y, z)^2 \dx\,\dy\,\dz + o(p^2) = \theta_1 + o(p^2), \end{align*} where we have applied \cref{cor:holder-bounded}.
\end{proof}
\begin{lemma}[Analysis of $t(E_3, U)$]
\label{lemma:t}
Let $U:V(G)^3 \to [0, 1]$ be a symmetric function satisfying \cref{eq:ipbound2}.
For parameters $b_1$ and $b_2$ satisfying $p^{2\eps_0}\ll b_1,b_2 \ll p^{\eps_0}$ we have that $t(E_3, U) \le \theta_2^{3/2} + o(p^3).$
\end{lemma}
\begin{proof}
We have that \begin{equation} \label{eq:t} t(E_3, U) = \int U(x, y, z')U(x, z, y')U(y, z, x') \dx\,\dy\,\dz\,\dx'\,\dy'\,\dz'. \end{equation}
We first claim that the contribution to \cref{eq:t} from the region with $x \in B_1$ is $o(p^3)$. The same would hold for $y, z$ by symmetry.
Indeed, the contribution from $x \in B_1$ is
\begin{align*} &\int_{B_1 \times V(G)^5} U(x, y, z')U(x, z, y')U(y, z, x') \dx\,\dy\,\dz\,\dx'\,\dy'\,\dz' \\ &\le \lambda(B_1)\int U(y, z, x') \dy\,\dz\,\dx' \ls \frac{p^2}{b_1} \cdot \frac{p^2 \log(1/p)}{\ka } = \frac{p^4 \log(1/p)}{b_1 \ka } \ll p^3, \end{align*} where we have used \cref{eq:lambound} and \cref{eq:l1bound}.

We now claim that the contribution to \cref{eq:t} from the region with $x \in B_3$ is $o(p^3)$. The same would hold for $y, z$ by symmetry. Indeed, the contribution from $x \in B_3$ is \begin{align*} &\int_{B_3 \times V(G)^5} U(x, y, z')U(x, z, y')U(y, z, x') \dx\,\dy\,\dz\,\dx'\,\dy'\,\dz' \\ &\le \int_{B_3} d(x)^2 \dx \ls \frac{p^3 b_2 \log(1/p)}{\ka } \ll p^3, \end{align*} where we have used \cref{eq:d2bound} for $b = pb_2$ and $b' = 0$.

Now, we consider the region where all of $x, y, z \in B_2$. In this case, we claim that the region where at least one of $x', y', z'$ lies in $\oB_3$ has contribution $o(p^3)$ to \cref{eq:t}. Without loss of generality, assume that $x' \in B_3.$ The contribution is
\begin{align*} &\int_{B_2^3 \times \oB_3 \times V(G)^2} U(x, y, z')U(x, z, y')U(y, z, x') \dx\,\dy\,\dz\,\dx'\,\dy'\,\dz' \\ &\le \lambda(B_2)^3 \lambda(\oB_3) \ls \left(\frac{p \log(1/p)}{\ka b_2} \right)^3 \cdot \frac{p \log(1/p)}{\ka b_2} \ll p^3, \end{align*} where we have used \cref{eq:lambound}.

Therefore, we have that
\begin{align*} 
t(E_3, U) &= \int_{B_2^3 \times B_3^3} U(x, y, z')U(x, y', z)U(x', y, z) \dx\,\dy\,\dz\,\dx'\,\dy'\,\dz' + o(p^3) \\ 
&\le \left(\int_{B_2^2 \times B_3} U(x, y, z)^2 \dx\,\dy\,\dz \right)^{3/2} + o(p^3) = \theta_2^{3/2} + o(p^3) \end{align*} 
after using \cref{cor:holder-bounded}.
\end{proof}

\begin{lemma}[Analysis of $t(H, U)$]
\label{lemma:h}
Let $U:V(G)^3 \to [0, 1]$ be a symmetric function satisfying \cref{eq:ipbound2}. There exist choices of $b_1$ and $b_2$ such that $p^{2\eps_0} \ll b_1, b_2 \ll p^{\eps_0}$ and \[ t(H, U) \le 3\theta_1^2 + 3\theta_2^2 + \theta_3^2 + 3\eta^2 + o(p^4). \]
\end{lemma}

Recall that \begin{equation} \label{eq:h} t(H, U) = \int U(x, y, z')U(x, z, y')U(y, z, x')U(x', y', z') \dx\,\dx'\,\dy\,\dy'\,\dz\,\dz'. \end{equation}
Our proof strategy is to bound the contribution of the integral depending on whether which of the $B_i$ each of $x, y, z, x', y', z'$ is in.

Our first claim is that there exists a choice of threshold $b_1$ in the definition of $B_1$ such that the contribution to \cref{eq:h} from the region with $x \in B_1$ and $x' \in \oB_1$ is $o(p^4).$ The same holds for the regions $y \in B_1, y' \in \oB_1$ and $z \in B_1, z' \in \oB_1$ by symmetry. We can show this via an adaptive thresholding argument, as done in \cite{BGLZ17}.

\begin{lemma}
\label{lemma:xb1}
Let $U:V(G)^3 \to [0, 1]$ be a symmetric function satisfying \cref{eq:ipbound2}. There is a choice of a parameter $b_1$ (which possibly depends on $U$) satisfying $p^{2\eps_0} \ll b_1 \ll p^{\eps_0}$ for some sufficiently small constant $\eps_0$ (say $\eps_0 = \frac{1}{100}$) such that if we define $B_1 = \{ x \in V(G) : d(x) \ge b_1 \},$ then \[ \int_{B_1 \times \oB_1 \times V(G)^4} U(x, y, z')U(x, z, y')U(y, z, x')U(x', y', z') \dx\,\dx'\,\dy\,\dy'\,\dz\,\dz' = o(p^4) .\]
\end{lemma}
\begin{proof}
Our proof is via an adaptive thresholding argument.
First, note that by \cref{cor:holder-bounded} we have that \[ t(H, U) \le \left(\int U(x, y, z)^2 \dx\,\dy\,\dz \right)^2 \ls p^4. \] Now, let $C$ be a constant so that $t(H, U) \le Cp^4$. It is sufficient to show that for any constant $c > 0$ that there is some choice of $b_1$ (which possibly depends on $U$) such that \[ \int_{B_1 \times \oB_1 \times V(G)^4} U(x, y, z')U(x, z, y')U(y, z, x')U(x', y', z') \dx\,\dx'\,\dy\,\dy'\,\dz\,\dz' \le cp^4. \] 

Set $M = \frac{3C}{c}$. Now, choose $b_1^{(1)}, b_1^{(2)}, \dots, b_1^{(M)}$ such that \[ p^{2\eps_0} \ll b_1^{(1)} \ll b_1^{(2)} \ll \dots \ll b_1^{(M)} \ll p^{\eps_0}. \]
Define \[ B_1^{(i)} = \{ x \in V(G) : d(x) \ge b_1^{(i)} \}. \]
Define the quantity \[ S_i = \int_{B_1^{(i)} \times \nbone \times V(G)^4} U(x, y, z')U(x, z, y')U(y, z, x')U(x', y', z') \dx\,\dx'\,\dy\,\dy'\,\dz\,\dz' \text{ for } 1 \le i \le M. \] We want to show that there is some $1 \le i \le M$ such that $S_i \le cp^4.$ Assume for contradiction that $S_i \ge cp^4$ for all $1 \le i \le M.$ Now, note that
\begin{align*} &\int_{B_1^{(i+1)} \times \nbone \times V(G)^4} U(x, y, z')U(x, z, y')U(y, z, x')U(x', y', z') \dx\,\dx'\,\dy\,\dy'\,\dz\,\dz' 
\\ &\le \lambda(B_1^{(i+1)}) \int_{\nbone} d(x')^2 \dx' \ls \frac{p^2}{b_1^{(i+1)}} \cdot b_1^{(i)} p^2 = \frac{b_1^{(i)}}{b_1^{(i+1)}} \cdot p^4 \ll p^4 \end{align*} by \cref{eq:lambound}, \cref{eq:d2bound1}, and $b_1^{(i)} \ll b_1^{(i+1)}$ as chosen above.

This gives us that \begin{align*} &\int_{B_1^{(i)} \times \left(B_1^{(i-1)} \bs B_1^{(i)}\right) \times V(G)^4} U(x, y, z')U(x, z, y')U(y, z, x')U(x', y', z') \dx\,\dx'\,\dy\,\dy'\,\dz\,\dz' \\ &\ge S_i - o(p^4) \ge \frac{1}{2}cp^4. \end{align*}
Now, as all the sets $B_1^{(i)} \times \left(B_1^{(i-1)} \bs B_1^{(i)}\right) \times V(G)^4$ are disjoint, we immediately get by summing over $1 \le i \le M$ that $Cp^4 \ge t(H, U) \ge \frac{1}{2}Mcp^4$, a contradiction to our choice of $M$.
\end{proof}

Now, fix the choice $b_1$ to satisfy the conditions of \cref{lemma:xb1}. Now, we bound the contribution from the region where $x \in B_1.$
\begin{lemma}
\label{lemma:b1contrib}
Consider a choice of $b_1$ such that $p^{2\eps_0} \ll b_1 \ll p^{\eps_0}$ and \[ \int_{B_1 \times \oB_1 \times V(G)^4} U(x, y, z')U(x, z, y')U(y, z, x')U(x', y', z') \dx\,\dx'\,\dy\,\dy'\,\dz\,\dz' = o(p^4) \] for $B_1 = \{ x \in V(G) : d(x) \ge b_1 \}.$ The contribution to the integral in \cref{eq:h} of the region where at least one of $x, x', y, y', z, z'$ lies in $B_1$ is at most $3\theta_1^2 + o(p^4)$.
\end{lemma}
\begin{proof}
By our choice of $b_1$ and \cref{lemma:xb1}, we only consider the region where $x' \in B_1$ also, as the region where $x \in B_1, x' \in \oB_1$ has contribution $o(p^4)$ to \cref{eq:h}.
We now claim that the contribution to \cref{eq:h} from the region where any one of $y, z, y', z' \in \oB_3$ is $o(p^4).$ To show this, without loss of generality say $y \in \oB_3$. We have that the contribution is
\begin{align*} &\int_{B_1^2 \times \oB_3 \times V(G)^3} U(x, y, z')U(x, z, y')U(y, z, x')U(x', y', z') \dx\,\dx'\,\dy\,\dy'\,\dz\,\dz' \\
&\le \lambda(B_1)^2 \lambda(\oB_3) \ls \left(\frac{p^2}{b_1} \right)^2 \cdot \frac{p \log(1/p)}{b_2 \ka} \ll p^4, \end{align*} where we have used \cref{eq:lambound}. Therefore, the contribution of the region $x \in B_1, x' \in B_1$ to \cref{eq:h} is 
\begin{align*} &\int_{B_1^2 \times B_3^4} U(x, y, z')U(x, z, y')U(y, z, x')U(x', y', z') \dx\,\dx'\,\dy\,\dy'\,\dz\,\dz' + o(p^4) 
\\ &\le \left(\int_{B_1 \times B_3^2} U(x, y, z)^2 \dx\,\dy\,\dz \right)^2 + o(p^4) = \theta_1^2 + o(p^4) \end{align*} by \cref{cor:holder-bounded}. We get a total of $3\theta_1^2 + o(p^4)$ from the symmetric cases ($y, y' \in B_1$ and $z, z' \in B_1$).
\end{proof}

From now on, we can restrict ourselves to the region where none of $x, x', y, y', z, z'$ lies in $B_1$, as we have already bounded that contribution. We first focus on the cases where all $x, x', y, y', z, z'$ lie in $B_3.$
\begin{lemma}
\label{lemma:allb3}
The contribution to the integral in \cref{eq:h} of the region where all of $x, x', y, y', z, z'$ lie in $B_3$ is at most $\theta_3^2$.
\end{lemma}
\begin{proof}
If all $x, x', y, y', z, z' \in B_3$, the contribution to \cref{eq:h} is
\begin{align*} &\int_{B_3^6} U(x, y, z')U(x, z, y')U(y, z, x')U(x', y', z') \dx\,\dx'\,\dy\,\dy'\,\dz\,\dz' 
\\ &\le \left(\int_{B_3^3} U(x, y, z)^2 \dx\,\dy\,\dz \right)^2 = \theta_3^2 \end{align*} by \cref{cor:holder-bounded}.
\end{proof}
Finally, we bound the contribution in the case where none of $x, x', y, y', z, z'$ are in $B_1$ and not all of $x, x', y, y', z, z'$ are in $B_3.$
Therefore, without loss of generality assume that $x \in B_2.$ We first argue there is a choice of $b_2$ such that the region where $x \in B_2$ and $x' \in B_3$, has contribution $o(p^4)$ to \cref{eq:h}.
\begin{lemma}
\label{lemma:xb2}
Let $U:V(G)^3 \to [0, 1]$ be a symmetric function satisfying \cref{eq:ipbound2}, and let $b_1$ be a parameter satisfying $p^{2\eps_0} \ll b_1 \ll p^{\eps_0}$ for some sufficiently small constant $\eps_0$ (say $\eps_0 = \frac{1}{100}$). There is a choice of a parameter $b_2$ (which possibly depends on $U$) satisfying $p^{2\eps_0} \ll b_2 \ll p^{\eps_0}$ such that if we define \[ B_2 = \{ x \in V(G) : b_1 \ge d(x) \ge pb_2 \} \text{ and } B_3 = \{x : V(G) : pb_2 > d(x) \}, \] then \[ \int_{B_2 \times B_3 \times V(G)^4} U(x, y, z')U(x, z, y')U(y, z, x')U(x', y', z') \dx\,\dx'\,\dy\,\dy'\,\dz\,\dz' = o(p^4) .\]
\end{lemma}
\begin{proof}
We use an adaptive thresholding argument similar to that of \cref{lemma:xb1}.
First, note that by \cref{cor:holder-bounded} we have that \[ t(H, U) \le \left(\int U(x, y, z)^2 \dx\,\dy\,\dz \right)^2 \ls p^4. \] Now, let $C$ be a constant so that $t(H, U) \le Cp^4$. It is sufficient to show that for any constant $c > 0$ that there is some choice of $b_2$ (which possibly depends on $U$) such that \[ \int_{B_2 \times B_3 \times V(G)^4} U(x, y, z')U(x, z, y')U(y, z, x')U(x', y', z') \dx\,\dx'\,\dy\,\dy'\,\dz\,\dz' \le cp^4. \]

Set $M = \frac{3C}{c}$. Now, choose $b_2^{(1)}, b_2^{(2)}, \dots, b_2^{(M)}$ such that \[ p^{2\eps_0} \ll \frac{\log(1/p)^2}{\ka ^2} b_2^{(i-1)} \ll b_2^{(i)} \ll p^{\eps_0} \forall 2 \le i \le M. \] We can do this as $\frac{\log(1/p)^{2M}}{\ka ^{2M}} \ll p^{\eps_0}$ for any constant $M$ and $\eps_0 > 0$ (recall that we fixed $\ka  = 1/\log(1/p)$, a slowly decaying function).
Define \[ B_2^{(i)} = \{ x \in V(G): b_1 > d(x) \ge pb_2^{(i)} \} \text{ and } B_3^{(i)} = \{ x \in V(G) : d(x) \le pb_2^{(i)} \}. \]
Define \[ S_i = \int_{B_2^{(i)} \times B_3^{(i)} \times V(G)^4} U(x, y, z')U(x, z, y')U(y, z, x')U(x', y', z') \dx\,\dx'\,\dy\,\dy'\,\dz\,\dz' \text{ for } 1 \le i \le M. \] We want to show that $S_i \le cp^4$ for some $i$. Assume for contradiction that $S_i \ge cp^4$ for all $1 \le i \le M.$
We have that
\begin{align*} &\int_{B_2^{(i)} \times B_3^{(i-1)} \times V(G)^4} U(x, y, z')U(x, z, y')U(y, z, x')U(x', y', z') \dx\,\dx'\,\dy\,\dy'\,\dz\,\dz' \\ &\ls \lambda(B_2^{(i)}) \int_{B_3^{(i-1)}} d(x')^2 \dx' \ls \frac{p \log(1/p)}{b_2^{(i)} \ka} \cdot \frac{p^3 \log(1/p)b_2^{(i-1)}}{\ka } = \frac{b_2^{(i-1)} \log(1/p)^2}{b_2^{(i)} \ka ^2} p^4 \ll p^4 \end{align*}
by \cref{eq:lambound}, \cref{eq:d2bound}, and $\frac{\log(1/p)^2}{\ka ^2} b_2^{(i-1)} \ll b_2^{(i)}$.
Therefore, \begin{align*} &\int_{B_2^{(i)} \times \left(B_3^{(i)} \bs B_3^{(i-1)}\right) \times V(G)^4} U(x, y, z')U(x, z, y')U(y, z, x')U(x', y', z') \dx\,\dx'\,\dy\,\dy'\,\dz\,\dz' \\ &\ge S_i - o(p^4) \ge \frac{1}{2}cp^4. \end{align*}
As the sets $B_2^{(i)} \times \left(B_3^{(i)} \bs B_3^{(i-1)}\right) \times V(G)^4$ are disjoint, summing over all $1 \le i \le M$ gives us that $Cp^4 \ge t(H, U) \ge \frac{1}{2}cMp^4$, a contradiction.
\end{proof}

We can now bound the contribution to \cref{eq:h}. of the region where none of $x, x', y, y', z, z'$ are in $B_1$ and not all of $x, x', y, y', z, z'$ are in $B_3$.
\begin{lemma}
\label{lemma:b2contrib}
Consider a choice of $b_2$ satisfying the constraints of \cref{lemma:xb2}. The contribution to the integral in \cref{eq:h} of the region where at least one of $x, x', y, y', z, z'$ lies in $B_2$ and none of $x, x', y, y', z, z'$ lie in $B_1$ is at most $3\eta^2 + 3\theta_2^2 + o(p^4).$
\end{lemma}
\begin{proof}
We apply \cref{lemma:xb2} and fix a threshold $b_2$ such that the contribution to \cref{eq:h} from the region $x \in B_2$ and $x' \in B_3$ (and the symmetric regions $y \in B_2$ and $y' \in B_3$ or $z \in B_2$ and $z' \in B_3$) is $o(p^4).$ We claim that if all of $x, x', y, y', z, z'$ lie in $B_2$ then the contribution to \cref{eq:h} is $o(p^4).$ Indeed, we have that \[ \int_{B_2^6} U(x, y, z')U(x, z, y')U(y, z, x')U(x', y', z') \dx\,\dx'\,\dy\,\dy'\,\dz\,\dz' \le \lambda(B_2)^6 \ls \left(\frac{p\log(1/p)}{b_2 \ka}\right)^6 \ll p^4 \] by \cref{eq:lambound}. The only remaining regions to analyze are
\begin{enumerate}
\item $x, x' \in B_2$ and $y, y', z, z' \in B_3$ (and its two symmetric versions)
\item $x, x', y, y' \in B_2$ and $z, z' \in B_3$ (and its two symmetric versions).
\end{enumerate}
This tells us that
\begin{align*}
&\int_{\oB_1^6} U(x, y, z')U(x, z, y')U(y, z, x')U(x', y', z') \dx\,\dx'\,\dy\,\dy'\,\dz\,\dz' \\
&\le  3\int_{B_2^2 \times B_3^4} U(x, y, z')U(x, z, y')U(y, z, x')U(x', y', z') \dx\,\dx'\,\dy\,\dy'\,\dz\,\dz' \\
&+ 3\int_{B_2^4 \times B_3^2} U(x, y, z')U(x, z, y')U(y, z, x')U(x', y', z') \dx\,\dx'\,\dy\,\dy'\,\dz\,\dz' + o(p^4) \\
&\le 3\left(\int_{B_2 \times B_3^2} U(x, y, z)^2 \dx\,\dy\,\dz \right)^2 + 3\left(\int_{B_2^2 \times B_3} U(x, y, z)^2 \dx\,\dy\,\dz \right)^2 + o(p^4) \\
&\le 3\eta^2 + 3\theta_2^2 + o(p^4)
\end{align*}
after using \cref{cor:holder-bounded}.
\end{proof}
\begin{proof}[Proof of \cref{lemma:h}]
Follows from \cref{lemma:xb1,lemma:b1contrib,lemma:allb3,lemma:xb2,lemma:b2contrib}.
\end{proof}
\begin{proof}[Proof of \cref{thm:varsolutions}(b)]
Recall that we have that
\begin{equation}\label{eq:done1} (1+\d)p^4 \le t(H, W) = p^4 + 4p^3 t(E_1, U) + 6p^2 t(E_2, U) + 4p^3 t(E_3, U) + t(H, U). \end{equation} Additionally, we have by \cref{lemma:x2} that \begin{align} \int I_p(p+U(x, y, z))\dx\,\dy\,\dz \nonumber &\ge (1-o(1))\log(1/p) \int U(x, y, z)^2 \dx\,\dy\,\dz \\ &\ge (1-o(1))\log(1/p)(3\theta_1 + 3\theta_2 + 3\eta + \theta_3). \label{eq:done2} \end{align} Combining \cref{eq:done1,eq:done2,lemma:e1,lemma:e2,lemma:t,lemma:h}
gives us that
\begin{align} &\frac{\phi(H, n, p, \d)}{n^3 \log(1/p)} \nonumber \\ &\ge (1-o(1))\frac{1}{3!} \inf\{ 3\theta_1 + 3\theta_2 + 3\eta + \theta_3 : 6p^2\theta_1 + 3\theta_1^2 + 4p\theta_2^{3/2} + 3\theta_2^2 + \theta_3^2 + 3\eta^2 \ge (\d-o(1))p^4 \}. \label{eq:done3} \end{align}
To bound the quantity in \cref{eq:done3}, note that because all the functions \[ 6p^2\theta_1 + 3\theta_1^2, 4p\theta_2^{3/2} + 3\theta_2^2, \theta_3^2, \text{ and }3\eta^2\] are convex, the infimum in \cref{eq:done3} is achieved when exactly one of $\theta_1,\theta_2,\theta_3$, and $\eta$ is nonzero. A direct computation using this observation shows that the right hand side of \cref{eq:done3} is at least
\[ (1-o(1)) \frac{p^2}{6} \cdot \min\left\{\sqrt{9+3\d}-3, 3\theta_2^\star, \sqrt{\d}, \sqrt{3\d} \right\} \] where $\theta_2^\star$ is the solution to $4(\theta_2^\star)^{3/2} + 3(\theta_2^\star)^2 = \d$,
and it is straightforward to check that $3\theta_2^\star \ge \min\left\{\sqrt{9+3\d}-3, \sqrt{\d} \right\}$. Hence the quantity in \cref{eq:done3} is at least \[ (1-o(1)) \cdot \frac{p^2}{6} \cdot \min\left\{\sqrt{9+3\d}-3, \sqrt{\d} \right\} \] as desired.
\end{proof}
In particular, this proof shows that solution to the variational problem is given by either planting several $1$-hubs or $3$-hubs. In the case where $1$-hubs are planted, this increases the number of copies of $E_2$, and corresponds to the $6c_1$ term in Equation \ref{eq:plant13}. In the case where $3$-hubs are planted, this directly increases the number of copies of $H$, and corresponds to the $c_{1/3}^6$ term in Equation \ref{eq:plant13}. The case of planting $2$-hubs was only excluded in the final computation, and a more intuitive explanation of this phenomenon would be interesting.

\subsection{Speculations about general hypergraphs}
Now we give some comments about possible approaches towards resolving \cref{conj:main}, and potential obstacles.
One possible approach resolves around making finer thresholds than those defined in \cref{sec:lowerbound}. Specifically, it is reasonable to define
thresholds \[ B_1 = \{x : b_1 \le d(x) \}, B_2 = \{x : b_2 \le d(x) < b_1 \}, \dots \] for well-chosen constants $b_1, b_2, \dots$. There are several potential obstacles. First, the method of performing the 
analysis of \cref{sec:lowerbound,sec:quad} relies on selecting for each vertex $i$ which set $B_j$ it lies in,
and then showing that most of these lead to negligible contributions. In this paper, we only had at most three sets $B_1,B_2,B_3$ so resolving all the cases was reasonable. When the number of sets we are considering increases,
there are significantly more cases to consider, and thus may be more difficult to analyze.
Additionally, there is no guarantee that this type of threshold is sufficient to show our desired conjecture. In fact, one
may need to use more general thresholds. Specifically, extend our ``degree" function $d$ to multiple inputs in the following way.
In the case of $r$-graphs, for an integer $k$ define
\[ d_k(x_1, x_2, \dots, x_k) = \int U(x_{[r]}) \dx_{k+1} \dots \dx_r. \] Then we could define sets
\[ B_b = \{(x_1, x_2, \dots, x_k) : b \le d_k(x_1, \dots, x_k) \} \] for example. Once again, this type of thresholding argument
becomes much harder to keep track of.

\appendix

\section{Large deviation principle for hypergraphs}
\label{app:ldp}
\subsection{Large deviations from Gaussian width}

We show how to deduce \cref{thm:fromeldanldp} from the large deviation principle of Eldan~\cite{Eld18}. One can also derive it from the earlier LDP of Chatterjee and Dembo~\cite{CD16} following ideas from \cite{BGSZ}. It is likely that we can apply more recent frameworks of \cite{CD} or \cite{HMS} to obtain better dependencies, though it would likely involve some highly non-trivial work. We do not try to optimize the parameters here and simply give the fastest way to derive some large deviations framework for random graphs that works for $p \ge n^{-c}$ for some constant $c = c_H > 0$. For constant or extremely slowly decaying $p$, one can alternatively apply the original Chatterjee--Varadhan method~\cite{CV11} with the hypergraph regularity theorem instead of the graph regularity theorem.

 We now make the necessary definitions in order to apply Eldan's result.
\begin{definition}
For a subset $K \subseteq \R^n$, define the \emph{Gaussian-width} of $K$ to be \begin{equation*} \GW(K) = \E_{\Gamma} \left[ \sup_{x \in K} \left\l x, \Gamma \right\r \right] \end{equation*} $\Gamma \sim N(0, \mathrm{Id})$ is a standard Gaussian random vector in $\R^n$.
\end{definition}
For a function $f: \{0, 1\}^N \to \R$, we can define the \emph{discrete derivatives} of $f$
\begin{equation*}
\DiscDerv_i(f(x)) = F(x_1, \dots, x_{i-1}, 1, x_{i+1}, \dots, x_N) - F(x_1, \dots, x_{i-1}, 0, x_{i+1}, \dots, x_N)
\end{equation*}
for any $1 \le i \le N$ and $x = (x_1, \dots, x_N) \in \{0,1\}^N$. 
Using this, we can define the \emph{discrete gradient} of $f$ as 
\begin{equation*}
\DiscGrad(f(x)) = (\DiscDerv_1(f(x)), \DiscDerv_2(f(x)), \dots, \DiscDerv_N(f(x))).
\end{equation*}
An important quantity in Eldan's large deviations result is the Gaussian width of the set of
discrete gradients of $f$, which we define as
\begin{equation*}
\DiscGW(f) = \GW( \{ \DiscGrad(f(x)) : x \in \{0, 1\}^N \} \cup \{0 \}).
\end{equation*}
Our application will rely on bounding the discrete Gaussian width of a counting function associated to a hypergraph, which we do in \cref{sec:gw}.
Now, define the \emph{discrete Lipschitz constant} of $f$ as
\begin{equation*}
\DiscLip(f) = \max_{1 \le i \le N, x \in \{0, 1\}^N} \DiscDerv_i(f(x))
\end{equation*}
We can define the natural variational problem associated to the function $f$ as
\begin{equation*}
\phi^f_p(t) = \inf_{x \in [0, 1]^N} \left\{ \sum_{i = 1}^N I_p(x_i) : \E f(X) \ge tN \right\}.
\end{equation*}
where the expectation is taken with respect to a random vector $X = (X_1, X_2, \dots, X_N)$ where $X_i \sim \text{Bernoulli}(x_i)$ independently for each $i \in [N].$

Eldan's results give us a large deviation principle provided that the Gaussian width can be efficiently controlled.

\begin{theorem}[Eldan \cite{Eld18} Theorem 5]
\label{thm:eldanldp}
Let $X = (X_1, X_2, \dots, X_N) \in \{0, 1\}^N$ be a random vector with i.i.d $X_i \sim \text{Bernoulli}(p).$ Given a function $f : \{0, 1\}^N \to \R$, for every $t, \eps \in \R$ with $0 < \eps < \phi_p^f(t-\eps)/N$, we have
\begin{equation*}
\log\Pr\left[ f(X) \ge tN \right] \le -\phi_p^f(t-\eps) (1 - 64LN^{-1/3})
\end{equation*}
with
\begin{equation*}
L = \frac{1}{\eps} \left(2\DiscLip(f) + |\log(p(1-p))| \right)^\frac23 \left(2\DiscGW(f) + \frac{1}{\eps} \DiscLip(f)^2 \right)^{1/3}.
\end{equation*}
Also if $\frac{1}{N\eps^2}\DiscLip(f)^2 \le \frac12$ then we have the lower bound
\[ \log\Pr\left[ f(X) \ge (t-\eps)N \right] \ge -\phi_p(t)\left(1 + \frac{2}{N\eps^2}\DiscLip(f)^2\right)-2. \]
\end{theorem}

Using this result, we can prove \cref{thm:fromeldanldp}, which reduces the upper tail problem to a discrete variational problem.
Before starting the proof, we define some notation.
We start by defining the function $T_H$ which counts the number of (not necessarily induced) copies of $H$ inside a graph. We intend to apply \cref{thm:eldanldp} on $T_H$.
\begin{definition}[Counting function associated to a hypergraph]
\label{def:counting}
Let $H$ be an $r$-graph and let $n$ be a positive integer. Define $T_H : \R^{\binom{n}{r}} \to \R$ by
\begin{equation*}
T_H(x) = \sum_{\substack{1 \le i_1, i_2, \dots, i_{|V(H)|} \le n \\ i_k \neq i_{k'} \text{ for } k \neq k'}} \prod_{(s_1, s_2, \dots, s_r) \in E(H)} x_{i_{s_1}, i_{s_2}, \dots, i_{s_r}},
\end{equation*}
where $x$ is invariant under permutation of coordinates, so that for all $\sigma \in S_r$ we have that $x_{i_1,\cdots,i_r} = x_{i_{\sigma(1)},\cdots,x_{i_{\sigma(r)}}}.$
\end{definition}

In order to apply Eldan's LDP on $T_H$, we must bound its discrete Gaussian width.
Note that $\DiscGrad(T_H(x)) = \nabla T_H(x)$ as $T_H(x)$ is linear in each variable. Therefore, we have that $\DiscGW(T_H) = \GW(\nabla T_H(\{0, 1\}^{\binom{n}{r}}))$, where \[ \nabla T_H(\{0, 1\}^{\binom{n}{r}}) := \{\nabla T_H(x) : x \in \{0, 1\}^{\binom{n}{r}} \}. \]
The following theorem provides the necessary bound on $\GW(\nabla T_H(\{0, 1\}^{\binom{n}{r}}))$.

\begin{theorem}[Upper bound on the Gaussian width of a $r$-graph]
\label{thm:gwupper}
Let $H$ be a $r$-graph, and let $n$ be a positive integer. Then we have that
\[ \frac{1}{n^{|V(H)| - r}} \cdot \DiscGW(T_H) \lesssim n^\frac{2r-1}{2}, \] where $T_H: \R^{\binom{n}{r}} \to \R$ is defined as in \cref{def:counting}.
\end{theorem}

We defer the proof to \cref{sec:gw}. Now, we combine the above results to show \cref{thm:fromeldanldp}.

\begin{proof}[Proof of \cref{thm:fromeldanldp}]
For our $r$-graph $H$ with maximum degree $\Delta$, in the statement of \cref{thm:eldanldp} set $N = \binom{n}{r}$, $t = (1+\delta)p^{|E(H)|}$, $X_1,\dots,X_N$ be the variables associated to the $N = \binom{n}{r}$ $r$-tuples in $\binom{[n]}{r}$, and $f(X) = N \cdot t(H, X) = N n^{-|V(H)|} T_H(X)$. We will choose $\eps$ later.

First, it is easy to see that \[ \DiscLip(f) \lesssim N \cdot \frac{n^{|V(H)|-r}}{n^{|V(H)|}} \lesssim 1. \] Also, by \cref{thm:gwupper}, we get that \[ \DiscGW(f) \lesssim N \cdot n^\frac{2r-1}{2} \cdot \frac{n^{|V(H)|-r}}{n^{|V(H)|}} \lesssim n^\frac{2r-1}{2}. \] Now, we restrict ourselves to $\eps$ satisfying \[ \eps \gtrsim n^{-\frac{2r-1}{2}} \gg N^{-1/2}. \] Now, a direct computation gives us (assuming $p \gg N^{-1}$ say) that \[ 64LN^{-1/3} \lesssim \eps^{-1} n^{-1/6} (\log N)^{2/3}. \] Because $p > n^{-\frac{1}{6|E(H)|}} \log n$, we can choose $\eps$ so that \[ p^{|E(H)|} \gg \eps \gg n^{-1/6}(\log N)^{2/3}. \] This way, for any fixed $\delta$, as $\eps \ll p^{|E(H)|}$, we know that \[ \phi_p^f(t-\eps) = \phi_p^f((1+\delta-o(\delta))p^{|E(H)|}) = (1+o(1)) \phi(H, n, p, \delta) \] as desired. Also, combining this with \cref{thm:asympbound} gives us that \[ \frac{64L(\log N)^{2/3}}{N^{1/3}} \ll 1 \text{ and } \eps < \phi_p^f(t-\eps)/N, \] where the second bound follows from \cref{thm:asympbound}: \[ \phi_p^f(t-\eps)/N \gtrsim \frac{n^r p^\Delta \log(1/p)}{N} \gtrsim p^\Delta \log(1/p) \gtrsim p^{|E(H)|} . \] This shows the upper bound. The lower bound follows from $\DiscLip(f) \le 1$ and $\eps \gg N^{-1/2}$.
\end{proof}

\subsection{The Gaussian Width of Hypergraphs}

\label{sec:gw}

In this section we show \cref{thm:gwupper}. First, we have the following easy bound derived from a union bound on Gaussian tails (see \cite[Lemma 4.5]{BGSZ}). 
\begin{lemma}[Small sets have small Gaussian width]
\label{lemma:gwsmall}

If $S \subseteq [-1, 1]^N$, then \[ \GW(S) \lesssim \sqrt{N \log |S|}. \]
\end{lemma}

Recall from above that $\text{DiscGW}(T_H) = \GW\left(\nabla T_H\left( \{ 0, 1\}^{\binom{n}{r}} \right) \right).$ Let the vertices of $H$ be indexed $1, 2, \dots, |V(H)|$ and let the edges be labelled $e_1, e_2, \dots, e_{|E(H)|}.$
For edge $e_i$, let the vertices in that edge be $s_{i, 1}, s_{i, 2}, \dots, s_{i, r}.$ Now, for functions $f_1, f_2, \dots, f_{E(H)} : [n]^r \to \R$, define
\begin{equation*} T(f_1, f_2, \dots, f_{|E(H)|}) = \sum_{1 \le i_1, i_2, \dots, i_{|V(H)|} \le n} \prod_{j=1}^{|E(H)|} f_j(i_{s_{j, 1}}, i_{s_{j, 2}}, \dots, i_{s_{j, r}}).
\end{equation*}

Consider an element $x \in \{0, 1\}^{n^r}.$ In this proof, we view $x$ instead as a function from $[n]^r$ to $\{0, 1\}$. In order to emphasize this view,
we use the notation $x(i_1, i_2, \dots, i_r) := x_{i_1, i_2, \dots, i_r}$ for $1 \le i_1, i_2, \dots, i_r \le n.$
Let $\Gamma$ denote a random $n^r$-dimensional Gaussian. Once again, we view $\Gamma$ as a function from $[n]^r \to \R.$ To emphasize this, we used the notation $\Gamma(i_1, \dots, i_r) := \Gamma_{i_1, \dots, i_r}.$ One can easily check that we have
\begin{equation}
\label{eq:done4}
\GW\left(\nabla T_H\left( \{ 0, 1\}^{\binom{n}{r}} \right) \right) \ls \E_\Gamma\left[ \sup_{x \in \{0, 1\}^{n^r}} \sum_{i = 1}^{|E(H)|} T(\overbrace{x, \dots, x}^{i-1 \text{ terms}}, \Gamma, x, \dots, x) \right].
\end{equation}
From here, our proof proceeds in two steps. First, we bound $T(x, \dots, x, \Gamma, x, \dots, x)$ by the \emph{cut-norm} of $\Gamma$, which we soon define. Afterwards, we use \cref{lemma:gwsmall} to bound the expected cut-norm of $\Gamma$. We now formally define the cut-norm.

\begin{definition}
\label{def:cutnorm}
Define the \emph{cut-norm} of function $f: [n]^r \to \R$ as
\begin{equation*}
\|f\|_{\square^r} = \sup_{\substack{u_1, u_2, \dots, u_r : [n]^{r-1} \to [0, 1] \\ u_1, u_2, \dots, u_r \text{ symmetric }}} \left| \sum_{1 \le i_1, \dots, i_r \le n} f(i_1, \dots, i_r) \prod_{k = 1}^r u_k(i_1, \dots, i_{k-1}, i_{k+1}, \dots, i_r). \right|
\end{equation*}
\end{definition}

Note that the expression inside the absolute value is linear in each of the $u_i$. Hence the supremum is achieved when the $u_i$ all have range $\{0, 1\}$ instead of $[0, 1].$ Specifically, we can write
\begin{equation*}
\|f\|_{\square^r} = \sup_{\substack{u_1, u_2, \dots, u_r : [n]^{r-1} \to \{0, 1\} \\ u_1, u_2, \dots, u_r \text{ symmetric }}} \left| \sum_{1 \le i_1, \dots, i_r \le n} f(i_1, \dots, i_r) \prod_{k = 1}^r u_k(i_1, \dots, i_{k-1}, i_{k+1}, \dots, i_r). \right|
\end{equation*}
instead. 

Before proceeding, we make one more observation, which we state now.
\begin{lemma}
\label{lemma:subset}
Let $f_1, f_2, \dots, f_\ell : [n]^r \to [0, 1]$ be functions that only depend on a proper subset of the coordinates, i.e. for all $j$ there exists an index $t$ such that $f_j(i_1, i_2, \dots, i_{t-1}, i_t, i_{t+1}, \dots, i_r)$ is constant over all $1 \le i_t \le n.$ Then we have that for any function $f: [n]^r \to \R$
\begin{equation*}
\left| \sum_{1 \le i_1, \dots, i_r \le n} f(i_1, \dots, i_r) \prod_{k = 1}^\ell f_k(i_1, \dots, i_r) \right| \le \|f\|_{\square^r}.
\end{equation*}
\end{lemma}

\begin{proof}
We will define disjoint sets $S_1, S_2, \dots, S_r \subseteq [\ell]$, where $\bigcup_k S_k = [\ell]$. Here, $S_k$ is essentially going to be the set of $j$ such that function $f_j$ doesn't depend on the $k$-th coordinate. Specifically, for each $f_j$, give it an index $t_j$ such that $f_j$ doesn't depend on the $t_j$-th coordinate. Now, define $S_k = \{ j : t_j = k \}.$ Now, define \[u_k(i_1, \dots, i_{k-1}, i_{k+1}, \dots, i_r) = \prod_{j \in S_k} f_j(i_1, i_2, \dots, i_{k-1}, 0, i_{k+1}, \dots, i_r), \] where we can put a $0$ in the $k$-th coordinate of each $f_j$ because it doesn't affect the value. Now, it is clear that
\begin{align*} &\left| \sum_{1 \le i_1, \dots, i_r \le n} f(i_1, \dots, i_r) \prod_{k = 1}^\ell f_k(i_1, \dots, i_r) \right| \\
= &\left| \sum_{1 \le i_1, \dots, i_r \le n} f(i_1, \dots, i_r) \prod_{k = 1}^r u_k(i_1, \dots, i_{k-1}, i_{k+1}, \dots, i_r) \right| \le \|f\|_{\square^r}
\end{align*}
by \cref{def:cutnorm}.
\end{proof}

Our next goal is to show that for all $1 \le i \le |E(H)|$
\[ \left|T(\overbrace{x, \dots, x}^{i-1 \text{ terms}}, \Gamma, x, \dots, x) \right| \le n^{|V(H)|-r} \|\Gamma\|_{\square^r}. \]
More generally, we show the lemma below, from which the above claim follows immediately.

\begin{lemma}
\label{lemma:cutbyabs}
For $f: [n]^r \to \R$ and functions $f_2, f_3, \dots, f_{|E(H)|} : [n]^r \to [0, 1]$ we have that
\[ \left| T(f, f_2, \dots, f_{|E(H)|}) \right| \le n^{|V(H)|-r} \|f\|_{\square^r}. \]
\end{lemma}

Before proving this, we prove it in the special case where $H = K_3$ as an illustrating example.

\begin{example}
\label{ex:k3}
Let $H = K_3.$ Then using the above definitions we can compute that
\[ T(f_1, f_2, f_3) = \sum_{1 \le x, y, z \le n} f_1(x, y)f_2(y, z)f_3(z, x). \]
Also, the cut-norm in this case is defined as
\[ \|f\|_{\square^2} = \sup_{\substack{u, v: [n] \to [0, 1] \\ u, v \text{ symmetric }}} \left| \sum_{1 \le x, y \le n} u(x)v(y) f(x, y) \right|. \]
Now, we can rewrite
\begin{align*}
|T(f, f_2, f_3)| &= \left| \sum_{1 \le x, y, z \le n} f(x, y)f_2(y, z)f_3(z, x) \right| \\
&\le \sum_{1 \le z \le n} \left| \sum_{1 \le x, y \le n} f_2(y, z)f_3(x, z)f(x, y) \right| \\
&\le \sum_{1 \le z \le n} \|f\|_{\square^2} = n\|f\|_{\square^2}.
\end{align*}
The inequality between the last two lines follows from the definition of the cut-norm: for a fixed $z$, we can define $u := f_3(x, z)$ and $v := f_2(y, z).$
\end{example}
Now we show \cref{lemma:cutbyabs} in general. The proof follows essentially the same format as above.
\begin{proof}[Proof of \cref{lemma:cutbyabs}]
Without loss of generality, assume that the first edge of $H$ contains vertices $1, 2, \dots, r.$
Then we can write
\[\left|T(f, f_2, \dots, f_{|E(H)|}) \right| \\
= \left| \sum_{1 \le i_{r+1}, \dots, i_{|V(H)|} \le n} \left( \sum_{1 \le i_1, \dots, i_r \le n} f(i_1, \dots, i_r) \prod_{k = 2}^{|V(H)|} f_k(i_{s_{k, 1}}, \dots, i_{s_{k, r}}) \right) \right|.
\]
Note that after fixing $i_{r+1}, \dots, i_{|V(H)|}$ in the outer sum in the second line of the above expression, that each of the functions $f_k$ only depends on a proper subset of $i_1, \dots, i_r$ for $2 \le k \le |V(H)|.$ Therefore, by \cref{lemma:subset} we have that
\begin{align*}
&\left|T(f, f_2, \dots, f_{|E(H)|}) \right| \\
&= \left| \sum_{1 \le i_{r+1}, \dots, i_{|V(H)|} \le n} \left( \sum_{1 \le i_1, \dots, i_r \le n } f(i_1, \dots, i_r) \prod_{k = 2}^{|V(H)|} f_k(i_{s_{k, 1}}, \dots, i_{s_{k, r}}) \right) \right| \\
&\le n^{|V(H)|-r} \|f\|_{\square^r}
\end{align*}
as desired.
\end{proof}

Therefore, we also have that
\[ \left|T(\overbrace{x, \dots, x}^{i-1 \text{ terms}}, \Gamma, x, \dots, x) \right| \le n^{|V(H)|-r} \|\Gamma\|_{\square^r}. \]
Using this, we now get that
\begin{align*}
&\frac{1}{n^{|V(H)|-r}} \GW\left(\nabla T_H\left( \{ 0, 1\}^{n^r} \right) \right) \\
&= \frac{1}{n^{|V(H)|-r}}\E_\Gamma\left[ \sup_{x \in \{0, 1\}^{n^r}} \sum_{i = 1}^{|E(H)|} T(\overbrace{x, \dots, x}^{i-1 \text{ terms}}, \Gamma, x, \dots, x) \right] \\
&\lesssim \E_\Gamma\left[\|\Gamma\|_{\square^r}\right].
\end{align*}

\begin{lemma}
\label{lemma:cutgw}
$\E_\Gamma \left[ \|\Gamma\|_{\square^r} \right] \lesssim n^\frac{2r-1}{2}.$
\end{lemma}
\begin{proof}
In the notation of \cref{lemma:gwsmall}, let $S$ correspond to the set of functions $f : n^r \to \{0, 1\}$ representable in the form
\[ f(i_1, i_2, \dots, i_r) = \pm \prod_{k = 1}^r u_k(i_1, \dots, i_{k-1}, i_{k+1}, \dots, i_r) \] for functions $u_k : n^{r-1} \to \{0, 1\}$ and a choice of sign.
It is then clear that $|S| \le 2 \cdot 2^{r \cdot n^{r-1}}.$ Additionally, we clearly have that
\begin{equation*}
\E_\Gamma \left[ \|\Gamma\|_{\square^r} \right] \le \E_\Gamma \left[ \sup_{f \in S} \left| \l \Gamma, f \r \right| \right] = \GW(S) \lesssim \sqrt{n^r \log |S|} \lesssim n^\frac{2r-1}{2}.\qedhere 
\end{equation*}
\end{proof}
\begin{proof}[Proof of \cref{thm:gwupper}]
Combine \cref{eq:done4} with \cref{lemma:cutgw,lemma:cutbyabs}.
\end{proof}

\section*{Acknowledgements}
The authors would like to thank the anonymous referees for careful reading and helpful comments.

\end{document}